\documentclass{article}
\usepackage{amssymb}
\usepackage{amsmath}
\usepackage{amsthm}
\usepackage{epsfig}
\usepackage{mathrsfs}
\usepackage{xcolor}
\usepackage{graphicx}
\usepackage{xpatch}

\newtheorem{theorem}{Theorem}[section]
\newtheorem{lemma}[theorem]{Lemma}
\newtheorem{proposition}[theorem]{Proposition}

\newtheorem{remark}[theorem]{Remark}
\newtheorem{rk&ex}[theorem]{Remarks \& Examples}
\newtheorem{corollary}[theorem]{Corollary}

\def\NN{{\mathbb N}}

\def\RR{{\mathbb R}}

\def\P{{\mathcal P}}

\def\eps{\varepsilon}

\def\to{\rightarrow}

\def\supp{\mbox{supp}\,}

\def\vphi{\varphi}

\def\na{\nabla}
\def\pa{\partial}

\begin{document}
\title{On the relationship between 
the thin film equation and Tanner's law}
\author{M. G. Delgadino\thanks{Partially supported by EPSRC grant number EP/P031587/1 and the Research Impulse Award at Imperial College} \, and A. Mellet\thanks{Partially supported by NSF Grant DMS-1501067}}

\date{\small Department of Mathematics\\
University of Maryland \\
College Park MD 20742\\
USA}
 \date{}

\maketitle
\begin{abstract}
    This paper is devoted to the asymptotic analysis of a thin film equation which describes the evolution of a thin liquid droplet on a solid support driven by capillary forces.
    We propose an analytic framework to rigorously  investigate the connection between this model and Tanner's law \cite{tanner1979spreading} which claims: \textit{the edge velocity of a spreading thin film on a pre-wetted solid is approximately proportional to the cube of the slope at the inflection}. 
More precisely, we investigate the asymptotic limit of 
 the thin film equation when the  slippage coefficient is small and at an appropriate time scale (see Equation \eqref{eq:tfe00}). We   show that the evolution of the droplet can be  approximated by a moving free boundary model (the so-called quasi-static approximation) and we present some results pointing to the validity of Tanner's law in that regime.
 Several papers \cite{giacomelli2016rigorous,GO,glasner2003spreading} have previously investigated a similar connection between the thin film equation and Tanner's law either formally or for particular solutions.
Our main contribution is the introduction of a new approach to systematically study this problem
by finding an equation for the evolution of the apparent support of the droplet (described mathematically by a nonlinear function of the solution). 
\end{abstract}

\section{Introduction}

\paragraph{The lubrication approximation and the  thin film equation.}
The thin film equation  describes the spreading of small liquid droplets on a solid (planar) surface
in the context of the lubrication approximation. 
This approximation is valid for thin droplets (the vertical dimension is much smaller than the horizontal one) when the evolution is dominated by the effects of surface tension and viscosity.

We briefly recall here the main steps of the derivation of the thin film equation. 
We refer to  \cite{Greenspan} for further details and for the physical dimensions of the coefficients (we also refer to \cite{GO2003,masmoudi1,masmoudi2} for a rigorous investigation of the lubrication approximation and derivation of the thin film equation for a liquid droplet in a Hele-Shaw cell).  
We present this derivation in dimension~$2$ even though
the rest of the paper will focus on the one dimensional case. 
We denote by $u(t,x)$ the height of the fluid (with $t>0$ and $x\in \RR^2$ or $\RR$) and by  $v(t,x,z)$ the horizontal velocity of the fluid ($z$ denotes the vertical variable).
The depth-averaged equation of mass conservation is
\begin{equation}\label{eq:continuity}
\frac{\pa u}{\pa t} + \mbox{div}_x(u\overline v)=0, \qquad \overline v(t,x) = \frac{1}{u(t,x)}\int_0^{u(t,x)} v(t,x,z)\, dz
\end{equation}
where the horizontal velocity $v$ satisfies the simplified Stokes equation
\begin{equation}\label{eq:stokes} 
\mu \frac{\pa^2 v}{\pa z^2} = \na p,
\end{equation}
where $\mu$ is the viscosity of the fluid. The pressure $p(t,x,z)$ is assumed to depend only on the $x$ variable and is  determined by surface tension  along the free surface $z=u(t,x)$ of the drop. Approximating the mean-curvature by the Laplacian, we find:
\begin{equation} \label{eq:pressure} 
p(t,x) = -\sigma \Delta_x u(t,x),
\end{equation}
where $\sigma$ is  the \textit{interfacial tension}. Finally, the velocity of the fluid  $v$ 
satisfies $\frac{\pa v}{\pa z}|_{z=u(t,x)} = 0$ (no horizontal shear stress along the free surface) while 
along the solid support $\{z=0\}$ it satisfies a slip condition:
\begin{equation}\label{eq:slip} 
\kappa(u) \frac{\pa v}{\pa z} (t,x,0) = v(t,x,0), \qquad \kappa(u)=\Lambda u^{n-2},
\end{equation}
where $\Lambda$ is the \textit{slip coefficient}. The case $n=2$ is usually refered as the Navier Slip condition. Integrating \eqref{eq:stokes} and using \eqref{eq:slip} leads to the following Darcy's law for the averaged velocity:
$$
\overline v = -\frac{1}{\mu} \left(\frac{u^2}{3} + \kappa(u) u \right)\na_xp,
$$
which together with \eqref{eq:continuity} and \eqref{eq:pressure}  yields the thin film equation
\begin{equation}\label{eq:thinfilm}
\frac{\pa u}{\pa t} + \frac{\sigma}{3 \mu}  \mbox{div}_x( (u^3+3 \Lambda u^n )\na \Delta u)=0.
\end{equation}

The derivation of this equation makes it clear that \eqref{eq:thinfilm} holds over the support of the drop $\{u>0\}$, thus leading to a delicate free boundary problem since boundary conditions must be imposed on $\pa\{u>0\}$.
However, it is classical, as a first step, to assume that \eqref{eq:thinfilm} holds in a fixed domain $\Omega$. This amounts to assuming that the solid substrate is wet even where the height of the fluid is zero. This is known as the {\it complete wetting regime}, and it is also the setting of this paper.
In  dimension one, 
the mathematical analysis of the thin film equation in the complete wetting regime goes back to the early 90's (see \cite{BF90}, \cite{BP96} for the one dimensional case. For results
in higher dimensions, we refer for instance to the contributions of Gr\"un \cite{Grun95,Grun04}).

We recall   (see  Bernis \cite{Bernis1,Bernis2} and Gr\"un \cite{Grun02,Grun03})  that the solutions of \eqref{eq:thinfilm} have a finite speed of propagation of the support (in particular solutions
with compactly supported initial data remain compactly supported) and that they satisfy the zero contact angle condition $|\na u |= 0$ on $\pa\{u>0\}$.

\paragraph{Previous results.}
When the \textit{slip coefficient} vanishes ($\Lambda=0$), the boundary condition \eqref{eq:slip} reduces to the classical no-slip condition $v(t,x,0)=0$. However, the corresponding thin film equation \eqref{eq:thinfilm}, with a mobility coefficient given by $u^3$, is classically ill-posed in the sense that the motion of the contact line $\pa\{u>0\}$ would require an infinite dissipation of energy (this fact was first formally discussed in \cite{HS}). This suggests that when $\Lambda$ is small, the contact line moves very slowly. 
It is the goal of this paper to investigate the asymptotic behavior of the solutions of \eqref{eq:thinfilm} when $\Lambda\ll1$. We thus introduce a small parameter $\eps$ such that:
\begin{equation}\label{eq:lambda}
\Lambda = \eps^{3-n}, \qquad \eps\ll1.
\end{equation}
Several work   (\cite{giacomelli2016rigorous,GO,glasner2003spreading}) have  investigated similar regimes in one dimension and  it has been shown that 
in order to observe the motion of the contact line at finite speed when $\eps$ goes to zero, we have to rescale time as follows:
\begin{equation}\label{eq:timescale}
t ' = t \frac{\sigma}{3 \mu}  |\ln \eps|^{-1}.
\end{equation} 
With \eqref{eq:lambda} and after change of variable \eqref{eq:timescale}, equation \eqref{eq:thinfilm} 
becomes in the one dimensional case:
\begin{equation}\label{eq:tfe00}
 |\ln(\eps)|^{-1} \pa_t u + \pa_x((\eps^{3-n} u^n +u^3) \pa_{xxx}u )=0.
\end{equation}
This paper is thus devoted to studying the behavior of the solutions of this equation when $\eps\ll1$.

\medskip

\begin{remark}\label{rem:1}
The  time scale \eqref{eq:timescale} is established by K. Glasner in \cite{glasner2003spreading} via energy consideration.
A more precise result is obtained for particular solutions by  L. Giacomelli and F. Otto in \cite{GO} and  by L. Giacomelli, M. Gnann and F. Otto in \cite{giacomelli2016rigorous}.
To state the result of \cite{GO} in our framework, we first note that if $u^\eps(t,x)$ is a solution of \eqref{eq:tfe} then the rescaled function $h$ defined by
\begin{equation}\label{eq:GOscalingeps}
    u^\eps(t,x) = \eps h(\eps^7 |\ln\eps| t,\eps x)
\end{equation}
solves
\begin{equation}\label{eq:GOtfe}
\pa_t h + \pa_x((h^n+h^3)\pa_{xxx}h)=0, \qquad h^\eps_{in}(x) = \frac{1}{\eps} u_{in}(\frac x\eps).    
\end{equation}
Using this, it is easy to check that Corollary 2.1 in  \cite{GO}  states that
if the initial condition has the form
$ u_{in}^\eps(x) = \eps h_{in}(\eps x)$
for some fixed function $h_{in}(x)$ (that is when $u_{in}^\eps$ is a very thin droplet), then for $\eps$ small enough, the solution $u^\eps$ of  \eqref{eq:tfe00} satisfies
\begin{equation}\label{eq:GO}
 C^{-1}\left(  t \frac{|\ln\eps|}{\ln t + \ln(\eps^7|\ln\eps| )}  \right)^{1/7} \leq  |\{u^\eps (t,\cdot)>\eps\}| \leq C\left(  t \frac{|\ln\eps|}{\ln t + \ln(\eps^7|\ln\eps| )}  \right)^{1/7}
\end{equation}
for time 
$ C\leq t  \leq \frac{C}{\eps^7 |\ln\eps|}.$
This result shows that for small $\eps$, the effective support of the drop (described here as the set where $\{u^\eps> \eps\}$) will grow with finite speed (and will have size of order $t^{1/7}$).
\end{remark}

\paragraph{The quasi-static model and Tanner's law.}
Note that without performing the change of variable  \eqref{eq:timescale} 
 the scaling of equation  \eqref{eq:tfe00} also corresponds to 
 $$ \frac{\sigma}{\mu} \sim |\ln\eps| \gg1 .$$
So equation  \eqref{eq:tfe00}   describes the asymptotic regime corresponding to capillary forces that are much stronger than viscous forces.
In that regime, we expect  the free surface of the drops to quickly relax to its equilibrium shape given by $\na p=0$ (this will be made rigorous using the dissipation of energy in the proofs).
With the approximation \eqref{eq:pressure}, we find
\begin{equation*} 
-\Delta u(t) = \lambda(t) \mbox{ in } \{u(t)>0\}, 
\end{equation*}
where $\lambda(t)$ can be seen as a Lagrange multiplier for the volume constraint and is locally constant on $\{u(t)>0\}$.
Remark \ref{rem:1} above suggests that the contact line $\pa\{u(t)>0\}$ is moving with finite speed, and the issue at the center of this paper is to determine the velocity law for this contact line.

More precisely we want to show that the regime  $\eps \ll1$ for the thin film equation \eqref{eq:tfe00} is described by a  moving free boundary problem of the form
$$\left\{
\begin{array}{ll}
-\Delta u(t) = \lambda(t) & \mbox{ in } \{u(t)>0\} \\[5pt]
 V = F(|\na u|)& \mbox{ on }\pa \{u(t)>0\}
\end{array}
\right.
$$
where $V$ denotes the normal velocity of the moving boundary $\{u(t)>0\}$ and $ \lambda(t)$ is the Lagrange multiplier.
In the context of the moving contact line problem for capillary drops, this equation is referred to as the quasi-static model and it is classically used to describe the motion of liquid drops when capillary forces dominate the dynamic in the bulk and intermolecular forces at the contact line are responsible for the motion of the drop.
In particular the velocity of the contact line is a function of 
the gradient $|\na u|$ on $\pa \{u(t)>0\}$ which is a substitute for the contact angle.
Several possible choice for the function $F$ have been proposed. In the complete wetting regime, the choice
\begin{equation}\label{eq:tannerlaw} 
V = \alpha |\na u|^3,
\end{equation}
for some coefficient $\alpha$, is motivated by the law proposed by Tanner in \cite{tanner1979spreading}.
We refer to \cite{GlasnerKim} for a mathematical analysis of this model.

In one space dimension, the connection between the thin-film equation and Tanner's law has first been rigorously studied in \cite{giacomelli2016rigorous}. In \cite{giacomelli2016rigorous}, the authors find a travelling wave solution of \eqref{eq:GOtfe} that satisfies Tanner's law with an appropriate logarithmic correction. We note that under the inverse scaling of \eqref{eq:GOscalingeps}, when we take $\eps\to0$ the travelling wave of \cite{giacomelli2016rigorous} converges formally to a travelling wave of the quasi-static approximation, which satisfies Tanner's law with no logarithmic correction.

\medskip

The main result of this paper will show that the evolution of the contact line is indeed approximated by Tanner's law \eqref{eq:tannerlaw} when one considers the  thin film equation in the complete wetting regime \eqref{eq:tfe00} with $\eps\ll1$ and in dimension one.
For simplicity, we consider equation \eqref{eq:tfe00} on a bounded set $\Omega\subset \RR$ and supplemented by  boundary and initial conditions:
\begin{equation}\label{eq:tfe}
 |\ln(\eps)|^{-1} \pa_t u + \pa_x((\eps^{3-n} u^n +u^3) \pa_{xxx}u )=0 \qquad \mbox{ in } (0,\infty)\times\Omega
\end{equation}
\begin{equation}\label{eq:bc}
(\eps^{3-n} u^n +u^3) \pa_{xxx}u  = 0 ,\qquad u_x =0 \qquad \mbox{on } (0,\infty)\times \pa\Omega
\end{equation}
$$ 
u^\eps(0,x)  = u_{in}(x) \qquad \mbox{ in } \Omega.
$$
The first boundary condition in \eqref{eq:bc} is a null flux condition, which will ensure conservation of mass. The second boundary condition can be seen as a contact angle condition for the fluid in contact with some boundary walls.

\paragraph{Apparent support and weak formulation of a free boundary problem.}
It is relatively simple to show that the limit $u$ of $u^\eps$ will satisfy $u_{xxx}=0$ in $\{u>0\}$. The difficulty is to characterize the motion of the support $\{u>0\}$.
The main contribution of this paper is the introduction of a new approach to study this problem. The idea  is to track the motion of the apparent support, which is approximately given by $\{u^\eps>|\ln(\eps)|^{-1}\}$, by  defining a specific sequence of smooth approximations of the indicator function $\chi_{\{s>0\}}$. To this end we introduce the following function $B$:
$$ B(s) := \int_0^s \int_r^\infty \frac{1}{v^{n-1}+v^2}\, dv\, dr.$$
With this definition, it is clear that $B(0)=0$ and 
$$ B''(s) = -\frac{1}{s^{n-1}+s^2}.$$

The function 
\begin{equation*}
B^\eps(s) = \frac{1}{|\ln\eps|} B(s/\eps)    
\end{equation*}
then satisfies
\begin{equation}\label{eq:propertyBeps1}
{B^\eps}''(s)=-\frac{1}{|\ln(\eps)|}\frac{1}{s^2+\eps^{3-n}s^{n-1}}
\end{equation}
and 
\begin{equation*}
B^\eps(0)=0, \qquad  \lim_{\eps\to 0}B^\eps(s)=1\mbox{ $\,\forall s>0$}.
\end{equation*}

To show this last property, we can note that 
$B'(s) = \int_s^\infty \frac{1}{v^{n-1}+v^2}\, dv \sim \frac{1}{s}$ as $s\to \infty$,
and so
$$\lim_{s\to \infty} s B'(s) =1.$$
L'Hospital's Rule then implies
$$ 
\lim_{\eps\to 0 } B^\eps (s) = \lim_{N\to \infty } \frac{B(N s)}{\ln N} =   \lim_{N\to \infty } Ns B'(N s)=  1
\qquad\mbox{for all $s>0$.}$$

\begin{remark}
We have the explicit formulas
$$
\begin{array}{rccr}
\displaystyle B^\eps(s)&=&\displaystyle\frac{1}{|\ln\eps|}\left[ \frac{s}{\eps} \arctan \left(\frac{\eps}{s}\right) + \frac 1 2\ln\left(1+\left(\frac{s}{\eps}\right)^2\right)\right]& \mbox{for $n=1$,}\\
\displaystyle  B^\eps(s)&=&\displaystyle\frac{1}{|\ln\eps|}\left[ \frac{s}{\eps} \ln \left(\frac{s+\eps}{s}\right) +\ln\left(1+\frac{s}{\eps}\right)\right]& \mbox{for $n=2$}.
\end{array}
$$
\end{remark}

Finally, given $u^\eps$ solution of \eqref{eq:tfe00}, 
we introduce the  function 
\begin{equation}\label{eq:rhoeps}
\rho^\eps(t,x)=B^\eps(u^\eps (t,x)).
\end{equation}
This function  can be viewed as an approximation of the characteristic function of the support of $u^\eps$. Indeed, we have 
\begin{equation*}
\rho^\eps(t,x)\approx\begin{cases}
0&\mbox{$u^\eps(t,x)\leq \eps$,}\\
1-a&\mbox{$u^\eps(t,x)\approx \eps^a$ with $a\in(0,1)$,}\\
1&\mbox{$u^\eps(t,x)\gg |\ln(\eps)|^{-1}$}.
\end{cases}
\end{equation*}
In particular, we have $\lim_{\eps\to 0}\rho^\eps (0,x) = \chi_{\{u_{in}>0\}}(x)$.
\medskip

The reason for the particular definition of $B^\eps(s)$, and the key to our approach, is the fact that
a straightforward computation using \eqref{eq:tfe} and \eqref{eq:propertyBeps1} leads to the following equation for $\rho^\eps$:
\begin{equation}\label{eq:tfed} 
\pa_t \rho^\eps = T^\eps+\pa_x R^\eps 
\end{equation}
where
\begin{equation}\label{eq:RTe}
\begin{array}{ll}
T^\eps  &\displaystyle = - u^\eps  u_x^\eps u_{xxx}^\eps\\[5pt]
R^\eps &\displaystyle = -|\ln(\eps)| B^{\eps\prime}(u^\eps)(\eps^{3-n}{u^{\eps}}^{(n-1)} +{u^{\eps}}^2) u^\eps u_{xxx}^\eps .
\end{array}
\end{equation}
Passing to the limit in \eqref{eq:tfed} will thus characterize the motion of the support of the drop in the limit $\eps\to0$.
We will easily show that $R^\eps$ goes to zero in an appropriate functional space, but the difficulty is to identify the limit of the distribution $T^\eps$ (since we expect $u_{xx}$ to be constant in $\{u>0\}$, it is reasonable to expect that $\lim T^\eps$ will be a distribution supported on $\pa\{u>0\}$).

\medskip

\paragraph{Main results.}
We now present our main results, Throughout this section, $u^\eps(t,x)$ denotes a weak solution of \eqref{eq:tfe00} satisfying some appropriate a priori estimates which are detailed in Section~\ref{sec:apriori}. The existence of such a solution is proved in Proposition \ref{prop:exist}.
\medskip

Classically, a first and crucial step in the study of singular limits of partial differential equations leading to free boundary problems is to establish the optimal regularity of the solutions uniformly in $\eps$.
In our case, since the gradient of the limit of $u^\eps$ is expected to be discontinuous at the boundary of the support $\{u>0\}$, the optimal regularity in space is the Lipschitz regularity.
We will prove:
\begin{proposition}[Lipschitz regularity in space]\label{prop:lip}
For all $\eps>0$, the solution $u^\eps(t,x)$ of \eqref{eq:tfe00} (whose existence is provided by Proposition \ref{prop:exist}) satisfies: 
\begin{equation}\label{eq:Lipu} 
 \| u^\eps_x \|_{  L^4	((0,\infty);L^\infty(\Omega))} \leq C
 \end{equation}
for a constant $C$ independent of $\eps$.
\end{proposition}

Next, we will prove the following result, which shows that the function $\rho^\eps$ defined by \eqref{eq:rhoeps} converges, up to a subsequence, to some function $\rho(t,x)$:
\begin{proposition}\label{prop:rho}
The followings hold uniformly with respect to $\eps$: 
$$ \rho^\eps \mbox{ is bounded in $L^\infty((0,T)\times\Omega)$}$$
and 
$$\rho^\eps _t \mbox{ is bounded in $L^1((0,T);W^{-1,1}(\Omega))$}.$$
In particular, $\{\rho^\eps\}_{\eps>0}$ is relatively compact in $L^p((0,T); W^{-1,1}(\Omega))$ for any $p\in[1,\infty)$ and for any  $T>0$.
\end{proposition}

We then study the properties of any accumulation point of $\{\rho^\eps\}_{\eps>0}$. The first result is the following:

\begin{theorem}\label{thm:rho}
Let $\rho(t,x)$ be an accumulation point of $\{\rho^\eps\}_{\eps>0}$ in $L^p(0,T; W^{-1,1}(\Omega))$. 
That is $\rho$ is such that
$$ 
\rho^{\eps_k} \to \rho \mbox{ strongly in $L^p(0,T; W^{-1,1}(\Omega))$}\qquad\mbox{for some subsequence $\eps_k\to0$.}
$$
Then the function $\rho(t,x)$ satisfies
\begin{equation}\label{eq:rho01}
 0\leq \rho(t,x)\leq 1\quad  \mbox{ a.e. in } (0,\infty)\times\Omega
 \end{equation}
and 
\begin{equation}\label{eq:positivity}
\pa_t \rho \geq 0  \quad \mbox{ in the sense of distribution.}
\end{equation}
\end{theorem}

The first inequality \eqref{eq:rho01}  follows straightforwardly from the definition of $\rho^\eps$ \eqref{eq:rhoeps} and a $L^\infty$ bound for $u^\eps$, which we derive in the next section.
The second inequality \eqref{eq:positivity} is considerably more delicate to obtain, and is really our first important contribution. 
If one thinks of $\rho$ as an approximation of the characteristic function of the support of the droplet, as suggested above, then \eqref{eq:positivity} implies that the support of the droplet is always expanding. 
This fact is expected, but as we will see, not  easy to establish.

\medskip

We  will not actually show that $\rho$ is a 
 characteristic function. This is unfortunately a  classical difficulty with capillary problems (\cite{GK10,mellet,otto98}). 
However, for any function $\rho(t,x)$ given by Theorem \ref{thm:rho}, 
the function $x\mapsto \rho(t,x)$ is well defined and belong to $L^\infty(\Omega)$ for almost every $t>0$.
For such $t$, we can then define the open set 
\begin{equation*} \Sigma (t) = \left\{x\in\Omega\,;\, \exists r>0\, \int_{B_r(x)\cap\Omega}[ 1-\rho(t,y)]\, dy =0 \right\}\end{equation*}
(this open set   is the complementary of the support of the positive measure $1-\rho$).
It satisfies in particular
\begin{equation*}
 \rho(t,x) = 1 \mbox{ a.e.  in } \Sigma(t)
\end{equation*}
(in fact we have $x_0\in \Sigma(t)$, if and only if, there exists $r>0$ such that $\rho(t,x)=1$ for almost every $x\in B_r(x_0)$)
and this definition makes it easy to show:
\begin{corollary}
The set $\Sigma(t)$ is non-decreasing:
$$ \Sigma(s) \subset \Sigma(t) \mbox{ whenever $0<s<t$.}$$
\end{corollary}

Intuitively, this set describes the support of the drop, a fact that will be made precise in the next theorem, which characterizes the behavior of the drop profile $\{u^{\eps_k}\}_{k\in\NN}$:
\begin{theorem}\label{thm:u}
For all $t>0$, $\{u^\eps(t,\cdot)\}_{\eps>0}$ is bounded in $H^1(\Omega)$ and therefore relatively compact in $C^0(\Omega)$.
Furthermore, given a subsequence $\eps_k\to 0$ such that 
$\rho^{\eps_k}$ converges to  $\rho$  strongly in $L^p(0,T; W^{-1,1}(\Omega))$,
there exists a subsequence $\eps_k'$ of  $\eps_k$ such that for almost every $t>0$, 
every accumulation point $v(x)\in C^0(\Omega) \cap H^1(\Omega)$ of the sequence $\{u^{\eps_k'}(t)\}$ satisfies:
\begin{equation}\label{eq:vpara}
v'''(x)=0 \mbox{ in } \{v>0\}, \qquad \int_\Omega v(x)\, dx=1
\end{equation}
and 
\begin{equation}\label{eq:vsupp}
\{v>0\} \subset \Sigma(t), \qquad \pa\{v>0\} \cap \Sigma(t) = \emptyset.
\end{equation}
\end{theorem}

Equation \eqref{eq:vsupp} implies that if $\Sigma(t)=\cup_{i\in I} (a_i,b_i)$, then $\{v>0\} = \cup_{i\in J}
(a_i,b_i)$ with $J\subset I$.
However \eqref{eq:vpara}-\eqref{eq:vsupp} do not fully identify the function $v$. For this, we would need to know how the volume of the drop is distributed among the connected components of $\Sigma(t)$. A simple case is when the open set $\Sigma(t)$ has a unique connected component, in which case there is a unique $v$ satisfying \eqref{eq:vpara}-\eqref{eq:vsupp} and the whole subsequence $u^{\eps_k'}(t,\cdot)$ converges to this function $v(x)$:
\begin{corollary}\label{cor:singleinterval}
For almost every $t>0$, if $\Sigma(t) = (a,b)$ has only one connected component compactly contained in $\Omega$, then 
the sequence $\{u^{\eps_k'}(t,\cdot)\}$ defined in Theorem \ref{thm:u} has only one accumulation point $v\in H^1(\Omega)$ supported in $(a,b)$, and satisfying
\begin{equation}\label{eq:vab} 
v''(x)=-\lambda \mbox{ in } (a,b), \quad \int_{(a,b)} v(x)\, dx =1.
\end{equation}
In particular, the sequence $u^{\eps_k'}(t,\cdot)$ converges uniformly and in $H^1$ weak to $v(x)$ for almost every $t>0$.
\end{corollary}

\begin{remark}\label{rem:oneintervaluniquelimit}
When $\Sigma(t) = (a,b)$, 
the function $v(x)$ is the unique minimizer of the following variational problem with volume constraint:
$$
\min_{w\in H^1_0(\Sigma(t)), \int w(x)\,dx=1}
\int_{\Sigma(t)} |w'(x)|^2\,dx
$$
and  it is explicitly given by the formula
$$
v(x)=6\frac{(b-x)_+(x-a)_+}{(b-a)^3}.
$$
\end{remark}
\begin{remark}\label{rem:uu}
Theorem~\ref{thm:u} formally shows that we cannot expect weak continuity in time for the accumulation points of $\{u^{\eps_k}\}$, which rules out any standard compactness argument and is one of the main mathematical obstacles to obtain the limit equation for $\pa_t \rho$.
 Indeed, the merging of droplets in \eqref{eq:tfe} takes place at a faster time scale than the spreading of the support. Our $\eps\to 0$ limit studies the latter time scale $1/|\ln\eps|$, see Remark~\ref{rem:1}. To illustrate this point explicitly, we formally examine the situation of two droplets merging at a time $t_0>0$.

If we assume that
$$
\begin{cases}
\Sigma(t)=(a_1(t),b_1(t))\cup(a_2(t),b_2(t))&\mbox{for $t<t_0$},\\
\lim_{t\to t_0^-}b_1(t)=\lim_{t\to t_0^-}a_2(t)&\mbox{for $t=t_0$},\\
\Sigma(t)=(a_1(t),b_2(t))&\mbox{for $t>t_0$},
\end{cases}
$$
then \eqref{eq:vpara} and \eqref{eq:vsupp} imply the following: For almost every $t<t_0$ there exists two positive functions $\gamma_1$ and $\gamma_2$, such that $\gamma_1(t)+\gamma_2(t)=1$ and
$$
v(t)=6\gamma_1(t)\frac{(b_1(t)-x)_+(x-a_1(t))_+}{(b_1(t)-a_1(t))^3}   +6\gamma_2(t)\frac{(b_2(t)-x)_+(x-a_2(t))_+}{(b_2(t)-a_2(t))^3}.
$$
Furthermore, by Remark~\ref{rem:oneintervaluniquelimit}, we have that for almost every $t\ge t_0$
$$
v(t)=6\frac{(b_2(t)-x)_+(x-a_1(t))_+}{(b_2(t)-a_1(t))^3}.
$$
Taking a test function $\phi$ depending on $\gamma_1$ and $\gamma_2$, and supported around the point $\lim_{t\to t_0^-}b_1(t)=\lim_{t\to t_0^-}a_2(t)$, we can obtain the discontinuity in time for the local average of the limit
$$
\lim_{t\to t_0^-} \int_\Omega v(t,x) \phi(x)\;dx\ne\lim_{t\to t_0^+} \int_\Omega v(t,x) \phi(x)\;dx.
$$
\end{remark}

Finally, our last result relates the evolution of the size of the support of the drop (which should be $|\Sigma(t)|$, but is represented in the theorem below by $\int \rho(t,x)\, dx \geq |\Sigma(t)|$)
and the profile of the drop:
\begin{theorem}\label{thm:sup}
Given a subsequence $\eps_k\to 0$ such that $\rho^{\eps_k}$ converges to  $\rho$  strongly in $L^p(0,T; W^{-1,1}(\Omega))$,
there exists a function 
$$w \in L^{\infty}(0,\infty; H^1(\Omega))\cap L^4(0,\infty; W^{1,\infty}(\Omega))$$ 
such that
 for a.a. $t\geq 0$
  $w(t,\cdot)$ is an accumulation point of the sequence $\{u^{\eps_k}(t,\cdot)\}_{k\in \NN}$ in $H^1(\Omega)$ satisfying \eqref{eq:vpara}, \eqref{eq:vsupp} and
\begin{equation}\label{eq:tannerglobal}
\displaystyle\frac{d}{dt} \int_\Omega \rho(t,x)  \, dx \geq
 \displaystyle\frac 1 3  \int_{\pa\{w>0\}} |w_x(t,x)| ^3  d \mathcal H^0(x)\qquad   \mbox{ in }  \mathcal{M}_+(0,T).
\end{equation}
\end{theorem}
This theorem, and in particular the inequality \eqref{eq:tannerglobal}, relates the evolution of the support $\Sigma(t)$ with Tanner's law \eqref{eq:tannerlaw}, though we only have an inequality (the support spread at least as fast as predicted by Tanner's law) and it is global in space. 
Note that \eqref{eq:tannerglobal} should be used together with an initial condition.
In general, we only have the following inequality (which follows from the monotonicity in time):
$$
\lim_{t\to0^+}\rho(x,t)\ge \rho_{in}=\lim_{\eps\to 0}B^\eps(u_{in})=\chi_{\{u_{in}>0\}}.
$$
This result should be compared with the result found in \cite{GO} (see Remark~\ref{rem:1}). Noticing that
$$
\int_{\Omega}\rho^\eps(t,x)\;dx\le|\{u^\eps>\eps\}|+C\frac{|\Omega|}{|\ln\eps|}
$$
we can use our theorem to recover the lower bound of \eqref{eq:GO} in the limit $\eps\to0$, with an explicit value for the constant.
\begin{corollary}\label{cor:ODE for the integral}
Under the assumptions and notations of Theorem~\ref{thm:sup},  we have the inequality
$$
\lim_{\eps\to 0} |\{u^\eps(x,t)\ge\eps\}| \geq \int_\Omega \rho(t,x)  \, dx\ge \min\{(63 \, t+|\{u_{in}>0\}|^7)^{1/7},|\Omega|\}\quad\mbox{for all $t\in[0, T]$.}$$

Moreover, if we have
\begin{equation*}
\pa \Sigma(t_0)\subset \Omega,    
\end{equation*}
then we have the stronger inequality
$$
\int_\Omega \rho(t,x)  \, dx\ge (1008\, t+|\{u_{in}>0\}|^7)^{1/7}\qquad\mbox{ for all $t\in[0, t_0]$}.
$$
\end{corollary}

\medskip

\paragraph{Formal derivation of Tanner's law.}
We end this introduction by outlining a formal derivation of Tanner's law and by pointing to the precise mathematical difficulties that currently prevents its full rigorous derivation. Formally, if the convergence is strong enough, we expect
$$
\lim_{\eps\to0} u^\eps= w\qquad\mbox{and}\qquad \lim_{\eps\to0}\rho^\eps=\rho=\chi_{\{w>0\}}.
$$ 
Moreover, by the energy dissipation inequality, we obtain
$$
w_{xxx}=0\qquad\mbox{on $\{w>0\}$,}
$$  
see \eqref{eq:vpara} and \eqref{eq:energy}. Next, passing to the limit distributionally in the time derivative equation for $\rho^\eps$ \eqref{eq:tfed} yields
$$ 
\pa_t \rho = \lim_{\eps\to 0} \left[ \pa_x R^\eps+T^\eps\right]
$$
with $R^\eps$ and $T^\eps$ given by \eqref{eq:RTe}.
We will show that $R^\eps$ converges to zero using the dissipation of energy, see Lemma~\ref{lem:Reps00}. Therefore, the motion of the contact line is characterized by the distribution 
$$ T :=  \lim_{\eps\to 0} T^\eps.$$
Integrating by parts, we notice that
\begin{align*}
\langle T^\eps(t) , \vphi\rangle_{\mathcal D',\mathcal D}  
& = - \int_\Omega u^\eps u^\eps_x u^\eps_{xxx} \vphi\, dx \\
& =  \int_\Omega u^\eps (u^\eps_{xx})^2 \vphi \, dx - \frac 5 {6}\int_\Omega {u_x^\eps}^3 \vphi' \, dx- \frac 1 2 \int_\Omega u^\eps {u^\eps_x}^2 \vphi''\, dx.
\end{align*}
Formally, passing to the limit in the previous equation, we get
$$
\langle T (t), \vphi\rangle_{\mathcal D',\mathcal D}   =  \int_{\{w>0\}} w (w_{xx})^2 \vphi \, dx - \frac 5 {6}\int_{\{w>0\}} {w_x}^3 \vphi' \, dx- \frac 1 2 \int_{\{w>0\}} w {w_x}^2 \vphi''\, dx
$$
and using that $w_{xxx}=0$ in $\{w>0\}$ and performing a couple of integration by parts, we obtain
\begin{equation}\label{eq:tannerdd}
\langle T(t) , \vphi\rangle_{\mathcal D',\mathcal D}   =  \frac 1 3  \int_{\pa\{w(t)>0\}} |w_x(t,x)|^3 \vphi(x)\, d\mathcal H^0(x) .
\end{equation}
The equation $\pa_t \rho = T$ is then a weak formulation of Tanner's  velocity law
$$ V = \frac1 3  |w_x|^3 .$$

\noindent There are two main difficulties to make this argument rigorous:

One is the lack of compactness in time for the function $u^\eps(t,x)$, as explained in Remark \ref{rem:uu}. This makes it necessary to work with limits for fixed times $t>0$, and  we need to be extremely careful with the type of a priori estimates we can use uniformly in time. 
The second reason, why we are only able to get an inequality in Theorem \ref{thm:sup} instead of the equality \eqref{eq:tannerdd}, is that in general we can only show that 
$$
\liminf_{\eps\to 0}\int_\Omega u^\eps (u^\eps_{xx})^2 \, dx \geq 
 \int_{\{w>0\}} w (w_{xx})^2 \, dx.
$$
To end this paper, we make this fact precise by proving a conditional result.
We consider the case of a single droplet spreading on a flat surface: We make the strong assumption that $\Sigma(t)$ has a single connected component for almost every $t>0$. Then Corollary \ref{cor:singleinterval} implies that 
the sequence  defined in Theorem \ref{thm:u}, which we denote $\{u^{\eps_k}(t,\cdot)\}$ for simplicity, converges to its unique accumulation point $ w(t,\cdot)\in H^1(\Omega)$, supported in $\Sigma(t)$ and solution of \eqref{eq:vab}.
We then have:
\begin{theorem}\label{thm:cond1}
In the simple framework described above, if we assume furthermore that
\begin{equation}\label{eq:cond0}
\lim_{k\to \infty}\int_0^T\int_\Omega u^{\eps_k} (u^{\eps_k}_{xx})^2\,   dx =
 \int_0^T\int_{\{w>0\}} w (w_{xx})^2\, dx,
\end{equation}
then the function $\rho(t,x) = \lim \rho^{\eps_k}$ satisfies
$ \rho (t,\cdot ) = \chi_{\Sigma(t)}$ and solves (in the sense of distribution):
\begin{equation}\label{eq:besteq}
\begin{cases}
\pa_t \rho = \frac 1 3  |w_x|^3  d \mathcal H^0|_{\pa\{w>0\}} \quad & \mbox{ in } \Omega\times (0,T) \\[3pt]
\rho(t=0) = \chi_{\{u_{in}>0\}} & \mbox{ in } \Omega.
\end{cases}
\end{equation}
\end{theorem}

\begin{remark}
If we denote $\Sigma(t) = (a(t),b(t))$, then using Remark \ref{rem:oneintervaluniquelimit}, \eqref{eq:besteq} implies
$$
\dot{a}=-\dot{b}=-72 (b-a)^{-6},
$$
for any $t<T$ as long as $a(T)$, $b(T)\notin \pa \Omega$. This  shows that the constant in Corollary~\ref{cor:ODE for the integral} is expected to be sharp.
\end{remark}
\begin{remark}
When $\Sigma(\cdot)$ consists of more than one interval, we can still identify a unique limit $w(t,x)$ if we know the amount of mass in each connected components of $\Sigma(t)$. 
This is possible if we know that mass cannot flow from one component to another. However,
we cannot discard such a phenomenon easily even when  these intervals are separated by  a positive distance. This phenomenon is usually referred to as Ostwald ripening, a mathematical analysis of the phenomenon for the thin film equation can be found in \cite{glasner2009ostwald}.
\end{remark}

\medskip

\paragraph{Organization of the paper}
Section~\ref{sec:apriori} collects the a priori estimates needed for the proof of the main results and Section \ref{sec:lip} establishes the Lipschitz regularity in space of the solution. Section~\ref{sec:rhoeps}, Section~\ref{sec:thmu} and Section~\ref{sec:thmsup} prove Theorem~\ref{thm:rho}, Theorem~\ref{thm:u} and Theorem~\ref{thm:sup} respectively. Section~\ref{sec:cond} 
contains the proof of the conditional result, Theorem \ref{thm:cond1},
and Section~\ref{sec:cor} proves Corollary~\ref{cor:ODE for the integral}.  
 Proposition~\ref{prop:exist} (which states the existence of a solution for $\eps>0$) is proved in Appendix~\ref{ap:exist}.

\medskip

\section{Important inequalities and a priori estimates}\label{sec:apriori}
In this section we present the main integral inequalities that will be used in the proofs. 
While the computations here are done formally without worrying about the regularity of $u^\eps$, the existence of  a solution satisfying the appropriate inequalities is stated at the end of this section (and proved in Appendix \ref{ap:exist}).

\paragraph{Conservation of mass and $L^1$ estimate.}
Integrating equation \eqref{eq:tfe} and using the null flux boundary condition \eqref{eq:bc} yields the following mass conservation equality:
\begin{equation}\label{eq:mass}
 \int_\Omega u^\eps(t,x)\, dx =   \int_\Omega u_{in}(x)\, dx \quad \forall t>0.
 \end{equation}
In particular,  we have
$$  \| u^\eps \|_{  L^\infty((0,\infty);L^1(\Omega))} \leq C.$$

\paragraph{Energy inequality and $H^1$ estimate.}
Classically, solutions of \eqref{eq:tfe} also satisfy the following energy inequality:
\begin{align}
  \int_\Omega \frac1  2 |u_x^\eps(t,x)|^2\, dx 
 + \int_0^t\int_{\{u^\eps>0\}} |\ln(\eps)| (\eps^{3-n} {u^\eps}^n +{u^\eps}^3) |u^\eps _{xxx}  |^2\, dx\, ds\nonumber \\
 \leq
\int_{\Omega} \frac 1 2 |\pa_x u_{in}(x)|^2  \, dx \qquad \forall t>0,\label{eq:energy}
\end{align}
which is obtained by multiplying \eqref{eq:tfe} by $-u^\eps_{xx}$ and integrating.
Together with the $L^1$ estimate above, this implies:
\begin{equation*} 
\| u^\eps \|_{  L^\infty((0,\infty);H^1(\Omega))} \leq C
\end{equation*}
which gives in particular
\begin{equation}\label{eq:uLinfty}
\sup_{x\in \Omega, t>0} u^\eps(t,x) \leq C.
\end{equation}

\paragraph{Entropy inequalities}
Solutions of \eqref{eq:tfe}-\eqref{eq:bc} 
also satisfy so-called entropy inequalities.
In this paper, we will make heavy use of the following entropy-like inequality
which follows from the choice of $B^\eps$ (see \eqref{eq:propertyBeps1}) and can be proved by integrating \eqref{eq:tfed} over $(0,T)\times\Omega$:
\begin{equation}\label{eq:entropy0}
\int_\Omega \rho^\eps(t,x)\;dx\geq \int_0^t \int_\Omega u^\eps |u^\eps_{xx}|^2\;dx\,ds
+ \int_\Omega \rho^\eps(0,x)\;dx .
\end{equation}
Note that this entropy is increasing  - one might thus prefer to think of $1-\rho^\eps$ as the entropy, but this quantity is not non-negative for fixed $\eps>0$, even though it is non-negative in the limit $\eps\to0$.

Since $s\mapsto B^\eps(s)$ is increasing, \eqref{eq:uLinfty} implies
\begin{equation}\label{eq:rhoeps01}
0\leq \rho^\eps(t,x)  = B^\eps(u^\eps(t,x)) \leq B^\eps(\|u^\eps\|_{L^\infty}) \leq B^\eps(C) \mbox{ for all $(t,x)$.}
\end{equation}
The right hand side of this inequality converges to $1$ as $\eps$ goes to zero (uniformly in $t$ and $x$), so we deduce
$$ 
\sup_{x\in \Omega, t>0} \rho^\eps(t,x) \leq C, \quad \mbox{ and }  \lim_{\eps\to0} \sup_{x\in \Omega, t>0} \rho^\eps(t,x) \leq 1.
$$
Using this bound and the fact that $\rho^\eps(0,x)\geq 0$, the inequality \eqref{eq:entropy0} yields
\begin{equation} \label{eq:entropydiss}
\int_0^\infty \int_\Omega u^\eps |u^\eps_{xx}|^2\;dx \leq  C|\Omega|.
\end{equation}

The inequalities above can easily be derived from \eqref{eq:tfe} if we assume that $u^\eps$ is smooth enough (in which case we get equalities). In Appendix~\ref{ap:exist}, we prove:

\begin{proposition}\label{prop:exist}
Let $n\in(0,3)$.
For all $\eps>0$, there exists a non-negative weak solution $u^\eps(t,x)$ of \eqref{eq:tfe} in  $L^\infty((0,T);H^1(\Omega))$ such that
\begin{equation}\label{eq:uuxxx is a function}
    u^\eps u^\eps_{xxx}\in L^2((0,T)\times\Omega) \qquad \mbox{ for all } \eps>0
\end{equation}
and the equations \eqref{eq:mass}, \eqref{eq:energy}, \eqref{eq:entropy0} hold.

Furthermore, the function $\rho^\eps(t,x)$ defined by \eqref{eq:rhoeps} solves \eqref{eq:tfed} in the sense of distribution, with the function $T^\eps \in L^2((0,T); L^1(\Omega))$ satisfying
\begin{equation}\label{eq:vphiT} 
\langle T^\eps , \vphi\rangle_{\mathcal D',\mathcal D}   =  \int_\Omega u^\eps (u^\eps_{xx})^2 \vphi \, dx - \frac 5 {6}\int_\Omega {u_x^\eps}^3 \vphi' \, dx- \frac 1 2 \int_\Omega u^\eps {u^\eps_x}^2 \vphi''\, dx.
\end{equation}
for all $\vphi \in \mathcal D(\overline \Omega)$.
\end{proposition}
\begin{remark}\label{rem:reg}
When $n\in  (0,2)$, the solution $u^\eps$ satisfies in addition that
$u^\eps\in  L^2((0,T);H^2(\Omega))$. In particular, $u^\eps (u^\eps_{xx})^2$ makes sense as a function in $L^1$ (and vanishes in the set $\{u^\eps=0\}$).
When $n\in [2,3)$, then (see \cite{BP96}) we only have ${u^\eps}^{3/2} \in  L^2((0,T);H^2(\Omega))$ and $({u^\eps}^{3/4})_x\in L^4$. Since we can write
$ u^{1/2} u_{xx} = \frac 2 3 (u^{3/2})_{xx} - \frac 89 ((u^{3/4})_x)^2$
this is again enough to make sense of the quantity 
$u^\eps (u^\eps_{xx})^2$ as an $L^1$ function for a.e. $t\in(0,T)$.

Note also that we can write
$u^\eps u^\eps_{xxx} = (u^\eps u^\eps_{xx})_x-\frac{1}{2}(u^{\eps2}_x)_x $ which makes sense  in all cases as a distribution. Proposition \ref{prop:exist} states that this distribution is in fact a function in $L^2$ for all $\eps>0$.
However, all these estimates degenerate when $\eps\to0$, which is why \eqref{eq:vphiT}  will be useful in the sequel.
\end{remark}

\section{Optimal (Lipschitz) regularity in space for all $t>0$.}\label{sec:lip}
In this section, we prove Proposition \ref{prop:lip}, which implies in particular that 
$u^\eps$ is Lipchitz uniformly in $\eps$ for almost every $t>0$.
The proof   relies on the following Lemma:
\begin{lemma}\label{lem:lipschitzbound}
For any $w\in C^1(\Omega)$ satisfying $w'=0$ on $\pa\Omega$,
there exists a constant $C$ depending only on $\|w\|_\infty $
such that for any $\eps\in(0,1)$ we have the inequality
\begin{equation*}
\|w'\|^2_{L^\infty}\le C \left(1+\int_{\{w>0\}} |\ln(\eps)| (\eps^{3-n} w^n +w^3) |w '''  |^2\, dx \right)^{\frac{1}{2}}.
\end{equation*}
\end{lemma}

Before proving the lemma, we note that Proposition \ref{prop:lip} follows by 
combining Lemma~\ref{lem:lipschitzbound} and the energy inequality \eqref{eq:energy}, provided the function $x\mapsto u^\eps(t,x)$ has the required $C^1$ regularity to apply Lemma \ref{lem:lipschitzbound} (at least for almost every $t$).
This regularity follows from Remark \ref{rem:reg}: 
When $n\in (0,2)$ the function $u^\eps(t,\cdot)$ is in $H^2(\Omega)$ for a.e. $t>0$ and therefore in $C^{1,1/2}(\Omega)$. When $n\in[2,3)$, we can write
$ \frac{1}{3}({u^\eps_x}^3)_x = {u_x^\eps}^2 u^\eps_{xx} = ({u^\eps}^{-1/2}{u_x^\eps}^2) {u^\eps}^{1/2}u^\eps_{xx}$, 
and Remark \ref{rem:reg} then implies
that ${u^\eps_x}^3 \in W^{1,1}(\Omega)$ for a.e. $t>0$. This implies in particular that $u^\eps_x$ is continuous, as needed.

\begin{proof}[Proof of Lemma~\ref{lem:lipschitzbound}]
We introduce the notations
$$M_\eps :=  \left( |\ln\eps| \int_{\{w>0\}} ( w^3+\eps^{3-n} w^n) |w '''  |^2\, dx\right)^{1/2}$$
and 
$$ G(x) := \frac{1}{2} {w'}(x)^2-w(x) w ''(x).  $$
If $M_\eps=\infty$ then the result follows trivially, so we can assume that  $M_\eps<\infty$. This implies in particular that $w\in H^3_{loc}(\{w>0\})$. A simple computation then gives
 $$ 
 G' = -w w '''\qquad\mbox{on $\{w>0\}$.}
 $$
We note that the $C^1$ regularity of $w$ and the boundary condition imply that
 $w'=0$ on $\pa\{w>0\}$.
 
Let $(a,b)\subset\{w>0\}$ be an interval on which $w'(x)$ has a constant sign (say $w'(x)\geq0$ for $x\in (a,b)$) and $w'=0$ at $a$ and $b$. Then we claim that 
\begin{equation}\label{eq:ineq2}
\int_a^b |G'| |w'|^{1/2}\, dx \leq C M_\eps
\end{equation}
for some constant $C$ independent of $\eps$ (and independent of $a$ and $b$).
To prove this, we first write (using H\"older inequality):
\begin{align}
\int_a^b |G'| |w'|^{1/2}\, dx 
& \leq  \left(\int_a^b w (w^2+\eps^{3-n} w^{n-1}) |w'''|^2  dx\right)^{1/2} 
 \left( \int_a^b \frac{w}{\eps^{3-n} w^{n-1} +w^2}|w'| \, dx\right)^{1/2} \nonumber \\
& \leq M_\eps \left(\frac 1{|\ln(\eps)| } \int_a^b\frac{w}{\eps^{3-n} w^{n-1} +w^2} |w'|\, dx\right)^{1/2}\label{eq:ineq1}
\end{align}
Furthermore, because $w'\geq 0$ on $(a,b)$, we can write\begin{align*}
\int_a^b\frac{w}{\eps^{3-n} w^{n-1} +w^2} |w'|\, dx & =\int_a^b\frac{w}{\eps^{3-n} w^{n-1} +w^2} w'\, dx\\  
& =\int_{w(a)}^{w(b)}\frac{s}{\eps^{3-n} s^{n-1} +s^2}\, ds\\
& \le \int_0^\eps\frac{1}{\eps^{3-n} s^{n-2}} \,ds+\int_\eps^{\|w\|_{\infty}}\frac{1}{s}\,ds\\
& \leq C + \ln(\|w\|_\infty)+|\ln \eps|,
\end{align*}
where we have used that $n<3$. Inserting this in \eqref{eq:ineq1}, we deduce \eqref{eq:ineq2}.
\medskip

Next, let $(c,d)$ be a subinterval of $(a,b)$ such that $w'\geq 1$ in $(c,d)$ (and $w'(x) =1$ for $x=c$ and $x=d$). If there are no such interval, then $w'\leq 1$ on $(a,b)$, and we are done.
If such an interval exists, then we note that \eqref{eq:ineq2} gives
\begin{equation}\label{eq:Gcd}
\int_c^d |G'| \, dx \leq CM_\eps.
\end{equation}

Furthermore, by our choice of the interval $(c,d)$ we have $w''(c)\geq 0$ and $w''(d)\leq 0$. So the definition of $G$ implies that $G(c) \leq \frac{1}{2} w'(c)^2 = \frac 1 2$ and  $G(d)\geq \frac{1}{2} w'(d)^2 = \frac 1 2$, which implies by continuity of $G$ that there exists a point $e\in (c,d)$, such that $G(e)=1/2$.

It follows (using \eqref{eq:Gcd}) that there exists a constant $C$ such that
$$ \| G\|_{L^\infty(c,d)}\leq C(1+M_\eps).$$
To conclude the argument, we note that $w'$ has a local maximum at some $x_0$ in $c,d$. At such a point, we have $w''(x_0)=0$ and so
$$ G(x_0 ) = \frac{1}{2} w'(x_0)^2.$$
We deduce that $w'(x)^2\leq 1+ \| G\|_{L^\infty(c,d)} \leq C(1+M_\eps)$ for all  $x\in(c,d)$.  Since we can repeat this argument on any such subinterval and the bound is independent of the size of the intervals involved, we deduce
$$ 
\|w'\|_{L^{\infty}(\Omega)}^2 \leq C(1+ M_\eps). 
$$
\end{proof}

\section{The function $\rho^\eps$ and its limit}\label{sec:rhoeps}
We now turn to the proof of our first main results.

\subsection{Proof of Proposition \ref{prop:rho}} 
We already know that $\rho^\eps$ is bounded in $L^\infty((0,\infty)\times\Omega)$ (see \eqref{eq:rhoeps01}). 
We thus only need  to show that $\pa_t \rho^\eps$ is bounded.  For that, we recall (see \eqref{eq:tfed}) that 
\begin{equation}\label{eq:rhoeps1} 
 \pa_t \rho^\eps = \pa_x R^\eps+T^\eps
\end{equation}
where
\begin{align*}
R^\eps &= -|\ln(\eps)| B^{\eps\prime}(u^\eps)(\eps^{3-n}{u^{\eps}}^{(n-1)} +{u^{\eps}}^2)u^\eps  u_{xxx}^\eps \\[5pt]
T^\eps  &= -u^\eps u_x^\eps u_{xxx}^\eps,
\end{align*}
where $T^\eps$ can also be written as in \eqref{eq:vphiT}.

We start with the following Lemma,
which together with \eqref{eq:energy} implies that
$$  \pa_x R^\eps \mbox{ is bounded in } L^2((0,T);W^{-1,1}(\Omega)).$$
\begin{lemma}\label{lem:Reps00}
There exists a constant $C$ (independent of $t$ and $\eps$) such that for all $t>0$, we have
\begin{equation}\label{eq:Rebd}
\int_\Omega |R^\eps(t,x)|\, dx \leq \frac{C}{|\ln\eps|^{1/2}} \left( \int_{\{u^\eps>0\}} |\ln\eps| (\eps^{3-n} {u^\eps}^n + {u^\eps}^3) |u^\eps_{xxx}|^2\, dx\right)^{1/2}.
\end{equation}
In particular, $R^\eps$ is uniformly bounded in $L^2((0,T);L^1(\Omega))$.
\end{lemma}
\begin{proof}
By the regularity \eqref{eq:uuxxx is a function} and H\"older's inequality, we have
\begin{align*}
\left( \int_\Omega |R^\eps(t,x)| \, dx\right)^2  
=\left( \int_{\{u^\eps>0\}} |R^\eps(t,x)| \, dx\right)^2  
& \leq \left(\int_{\{u^\eps>0\}} |\ln(\eps)|(\eps^{3-n} {u^\eps}^n+{u^{\eps}}^3) |\pa_{xxx}u^\eps|^2 \, dx\right)\\
& \qquad \times \sup_{t\in[0,T]}
\int_\Omega |\ln(\eps)|B'_\eps(u^\eps)^2 (\eps^{3-n} {u^\eps}^{n-1} +{u^{\eps}}^2)u^\eps \, dx.
\end{align*}
Using that 
$$
B_\eps'(s)=(\eps|\ln(\eps)|)^{-1}\int_{s/\eps}^\infty \frac{dr}{r^{n-1}+r^2}
$$ 
we deduce
\begin{align*}
|\ln(\eps)|B'_\eps(u^\eps)^2 (\eps^{3-n} {u^\eps}^{n-1} +{u^{\eps}}^2))&= \frac{1}{|\ln(\eps)|}\left(\int_{u^\eps/\eps}^\infty \frac{dr}{r^{n-1}+r^2}\right)^2\left(\left(\frac{{u^\eps}}{\eps}\right)^{n-1}+\left(\frac{{u^\eps}}{\eps}\right)^2\right)\\
&\le \frac{1}{|\ln(\eps)|}\sup_{s\in[0,\infty)}\left[ \left(\int_{s}^\infty \frac{dr}{r^{n-1}+r^2}\right)^2\left(s^{n-1}+s^2\right)\right]\\
&=\frac{C}{|\ln(\eps)|},
\end{align*}
where we have used that $n<3$. The uniform bound in $L^2((0,T);L^1(\Omega))$ follows from the energy dissipation inequality \eqref{eq:energy}.
\end{proof}

In order to bound the term $T^\eps$, we first write
\begin{align*}
T^\eps&=-u^\eps \pa_x u^\eps\pa_{xxx}u^\eps\nonumber \\
&=-\pa_x(u^\eps  u_x^\eps u_{xx}^\eps)+ | u_{x}^\eps|^2u_{xx}^\eps+u^\eps |u_{xx}^\eps|^2 \nonumber \\
&=\pa_x\left(\frac{( u_{x}^{\eps})^3}{3}- u^\eps u_x^\eps u_{xx}^\eps\right)+ u^\eps |u_{xx}^\eps|^2.
\end{align*}
Using \eqref{eq:entropydiss} and \eqref{eq:Lipu}, we easily deduce that
$$ T^\eps(t,x) \mbox{ is bounded in } L^1(0,T; W^{-1,1}(\Omega)).$$
This completes the proof of Proposition \ref{prop:rho} (by a classical application of Lions-Aubin's Lemma).
\begin{remark}
We notice that if we have the stronger hypothesis that
$$
\int_{\Omega} u^\eps |u_{xx}^\eps|^2\,dx \qquad\mbox{is uniformly integrable in $(0,T)$,}
$$
then $\{\rho^\eps\}_{\eps>0}$ is pre-compact in $C^0(0,T;W^{-1,1})$ and $\pa_t\rho\in L^1(0,T;W^{-1,1})$.
\end{remark}

\subsection{Proof of Theorem \ref{thm:rho}}
From now on, we fix a subsequence $\eps_k \to 0$ such that
\begin{equation*} 
 \rho^{\eps_k} \to \rho \quad \mbox{ strongly in  $L^p(0,T;W^{-1,1}(\Omega))$ and weak-* in $L^\infty$} 
 \end{equation*}
(such a subsequence exists thanks to Proposition \ref{prop:rho}).

\begin{remark}
It is possible to get a subsequence such that $\rho^{\eps_k}$ converges for all $t>0$, instead of a.e. $t>0$. This is done by noticing that 
$$
\pa_t\rho^{\eps_k}- u^{\eps_k} (u^{\eps_k}_{xx})^2\in L^p((0,T);W^{-1,1})\qquad\mbox{for some $p>1$}
$$
and that $u^{\eps_k} (u^{\eps_k}_{xx})^2$ is always positive. Therefore by applying Aubin-Lions Lemma for a part of $\rho^{\eps_k}$ and the pointwise convergence of monotone functions for the rest of $\rho^{\eps_k}$, we can obtain a pointwise limit for every $t\in(0,T)$.
\end{remark}

The first statement of Theorem \ref{thm:rho}, inequalities \eqref{eq:rho01}, follows immediately by 
passing to the limit in inequalities  \eqref{eq:rhoeps01}.
The second statement of the theorem (the fact that $\rho(t,x)$ is non-decreasing with respect to $t$) is more delicate and its proof will occupy the rest of this section.

\paragraph{Preliminary.}
First, we note that the energy inequality \eqref{eq:energy} implies that the
function
$$h^\eps(t):=   \int_{\{u^\eps>0\}}   (\eps^{3-n} {u^\eps}^n +{u^\eps}^3) |u^\eps _{xxx}  |^2\, dx$$
satisfies
$$ \int_0^\infty  h^\eps(t)\, dt \leq |\ln(\eps)|^{-1}\int_{\Omega} \frac 1 2 |\pa_x u_{in}(x)|^2  \, dx $$ and thus converges to $0$ strongly in $L^1(0,T)$.
So up to another subsequence (still denoted $\eps_k$), we can further assume that 
\begin{equation}\label{eq:limh}
h^{\eps_k}(t) \to 0 \qquad  \mbox{ for all } t\in\P', 
\end{equation}
where $\P'\subset\P\subset(0,T)$ is a set of full measure. Throughout the proof, the set of full measure that we work on gets progressively restricted. To avoid burdensome notation, we do not relabel the sets and it is always denoted by $\P$.

\medskip

We note that we construct the subsequence $\eps_k$ with the convergence of $\rho^{\eps_k}$ in mind, rather than that of $u^{\eps_k}$.
However,  for all $t>0$, $\{u^{\eps_k}(t)\}_{k\in \NN}$ is bounded in $H^1(\Omega)$ and therefore is pre-compact in $C^{0}(\Omega)$.
Using \eqref{eq:limh} we can then prove:
\begin{proposition}\label{prop:uacc}
Let $t\in \P$. Then any accumulation point $v\in C^{0}(\Omega) \cap H^1(\Omega)$ of the sequence $\{u^{\eps_k}(t)\}_{k\in \NN}$ satisfies
\begin{equation}\label{eq:parabola} v'''=0 \mbox{ in } \{v>0\}
\end{equation}
and 
\begin{equation}\label{eq:vmass} 
\int_\Omega v(x)\, dx = \int_\Omega u_{in}(x)\, dx.
\end{equation}
Furthermore, we have
\begin{equation}\label{eq:bcv}
\mbox{either $v=0$ or $v'=0$ on $\pa\Omega$.}
\end{equation}
\end{proposition}
\begin{proof}
We fix $t\in\P$ and
we recall that $H^1(\Omega) \subset C^{1/2}(\Omega)$ and so the functions $x\mapsto u^{\eps_k} (t,x)$ are equi-continuous with respect to $k$.
Let now $v(x)$ be such that
$$ u^{\eps_k'}(t,x)  \to v(x) \mbox{ uniformly in $x$ and weakly in $H^1(\Omega)$}.$$
The conservation of mass \eqref{eq:mass} immediately implies \eqref{eq:vmass}.
In order to prove \eqref{eq:parabola}, we first note that 
if $x_0\in \{v>0\}$ then there exists $\delta>0$ and $r$ such that
$$ u^{\eps_k'}(t,\cdot) \geq \delta \mbox{ in } B_r(x_0) \mbox{ and for all $k$ large enough}$$
Using \eqref{eq:limh}, we deduce
$$ 
\delta^3 \int_{B_r(x_0)} |u^{\eps_k'}_{xxx}|^2\, dx \leq  h^{\eps_k'}(t) \to 0$$
and so 
$$ v''' = 0 \mbox{ in } B_r(x_0).$$
Since this holds for all $x_0\in\{v>0\}$, the result follows.
It only remains to show \eqref{eq:bcv}, which follows from the fact that $u^\eps_x=0$ on $\pa\Omega$ and that if $v(x_0)>0$ for $x_0\in\pa \Omega$, then $u^{\eps_k'}_x$ converges in $C^2$ to $v$ in a neighborhood of $x_0$.
\end{proof}

\medskip

To prove that $\pa_t\rho$ is non-negative, we will naturally try to pass to the limit in the equation for $\pa_t\rho^\eps$ \eqref{eq:tfed}.
For that, we will show that $R^\eps$ goes to zero (in an appropriate sense) while the contribution of $T^\eps$ is always non-negative. 

\paragraph{The term $\pa_x R^\eps$ goes to zero}
We have the following lemma:
\begin{lemma}\label{lem:pReps}
For any smooth test function $\vphi\in \mathcal{D}(\overline{\Omega})$, we have
$$ \langle \pa_x R^{\eps_k}(t,\cdot), \vphi\rangle_{\mathcal D',\mathcal D} \to 0 \qquad \mbox{ for all $t\in\P$ and in $L^2(0,T)$}$$
\end{lemma}
\begin{proof}[Proof of Lemma \ref{lem:pReps}]
Using \eqref{eq:Rebd} we find
\begin{align*}
| \langle \pa_x R^{\eps_k}(t,\cdot), \vphi\rangle |
& \leq \| \vphi_x\|_{L^\infty} \int_\Omega |R^{\eps_k}(t,x)| \, dx \\
& \leq  \| \vphi_x\|_{L^\infty} C \left( \int_\Omega   (\eps^{3-n} {u^\eps}^n + {u^\eps}^3) |u^\eps_{xxx}|^2\, dx\right)^{1/2}\\
& \leq C \| \vphi_x\|_{L^\infty} h^{\eps_k}(t).
\end{align*}
To conclude, we simply note that the energy inequality \eqref{eq:energy} implies that $h^{\eps_k}$ converges to zero in $L^2(0,\infty)$, while the construction of the subsequence $\eps_k$ guarantees that $h^{\eps_k}(t)$ converges to zero for all $t\in\P$ (see \eqref{eq:limh}).
\end{proof}

\paragraph{The distribution $T^\eps$.}
To conclude the proof of Theorem \ref{thm:rho}, we will now show the following result:
\begin{proposition}\label{prop:T}
For any positive smooth test functions $\phi\in \mathcal{D}(\overline{\Omega})$ and $\eta\in \mathcal{D}([0,T])$, we have
\begin{equation*}
\liminf_{k \to \infty} \int \langle T^{\eps_k}(t) , \phi^4\rangle_{\mathcal D',\mathcal D} \, \eta(t)\, dt \geq 0.
\end{equation*}
\end{proposition}

The proof of Proposition \ref{prop:T} will occupy the rest of this section. The strategy to show Proposition \ref{prop:T} is to study the pointwise in time limit of $T^{\eps_k}(t)$. From now on, we further restrict the set of times $\mathcal{P}$, so that the representation for $T^{\eps_k}(t)$ \eqref{eq:vphiT} also holds for every $k$ and every $t\in\mathcal{P}$. We can then  prove:
\begin{proposition}\label{prop:T2}
For any positive test function $\phi\in \mathcal{D}(\overline{\Omega})$ we have
$$ \liminf_{k\to \infty} \langle T^{\eps_k}(t) , \phi^4\rangle_{\mathcal D',\mathcal D}  \geq 0 \quad \mbox{ for 
 all $t\in \P$. }$$
\end{proposition}
\medskip


A key result in the proof of this proposition is the following:
\begin{lemma}\label{lem:matias}
For all smooth positive test functions $\phi\in \mathcal{D}(\overline{\Omega})$, there exists a constant $C$ depending only on $\phi$ and $\int \left|\pa_x u_{0}(x)\right|^2\, dx$ such that
\begin{equation}\label{eq:matias}
\langle T^\eps(t), \phi^4 \rangle_{\mathcal D',\mathcal D} \geq \frac 1 4 \int_\Omega u^\eps(t) |u^\eps_{xx}(t)|^2 \phi^4 (x) \, dx - C
\end{equation}
for all $t\geq 0$ and for all $\eps>0$ such that \eqref{eq:vphiT} holds.
\end{lemma}
This lemma plays a crucial role in the proof of Proposition \ref{prop:T2}. It also shows that the function $t\mapsto \langle T^{\eps_k}(t), \phi^4 \rangle $ are bounded below uniformly (by $-C$).  We can thus use Fatou's lemma to show that Proposition \ref{prop:T2} implies Proposition \ref{prop:T}.

\medskip

The second important lemma for the proof of Proposition \ref{prop:T2} is the following:
\begin{lemma}\label{lem:uv}
For all $t\in[0,T]$, for all sequence $\eps_k\to 0$, and for all positive test function $\phi\in\mathcal{D}(\overline{\Omega})$, 
if there exists $v(x)$ such that
$$ u^{\eps_k}(t) \to v\qquad \mbox{ uniformly in $\Omega$}
$$
and 
\begin{equation}\label{eq:bd3dd}
 \int_\Omega u^{\eps_k}(t) (u^{\eps_k}_{xx}(t))^2\phi^4 \, dx \leq C
 \end{equation}
for some constant $C$, then
$$  u^{\eps_k}_x (t,x)\phi(x) \to v'(x) \phi(x) \qquad \mbox{ strongly  in $L^p(\Omega)$ for all $p<4$} .$$
\end{lemma}

Postponing the proof of this proposition to the end of this section, we now turn to the proof of Proposition \ref{prop:T2}:

\begin{proof}[Proof of Proposition \ref{prop:T2}]
Throughout the proof, we fix a test function $\phi\in\mathcal{D}(\overline{\Omega})$ and a time $t\in\P$.
If 
$ \liminf_{k\to \infty} \langle T^{\eps_k}(t) , \phi^4\rangle = \infty$, then the result is trivially true.
Otherwise,  there exists a subsequence $\eps_k'$ such that
$$\lim_{k\to \infty} \langle T^{\eps_k'}(t) , \phi^4\rangle = \liminf_{k\to \infty} \langle T^{\eps_k}(t) , \phi^4\rangle <\infty.$$
In particular 
$ \langle T^{\eps_k'}(t) , \phi^4\rangle$ is bounded and Lemma \ref{lem:matias} implies that 
\begin{equation}\label{eq:CCF} 
\int_\Omega u^{\eps_k'}(t) (u^{\eps_k'}_{xx}(t))^2\phi^4 \, dx < C
\end{equation}
for some constant $C$.

Since $u^{\eps_k'}(t)$ is also bounded in $H^1(\Omega)$, and thus in $C^{1/2}(\Omega)$, up to another a subsequence (still denoted $\eps_k'$), there exists a function $v\in H^1(\Omega)$ such that
$$ u^{\eps_k'}(t) \to v(x)\quad \mbox{ uniformly and in $H^1(\Omega)$ weak.}
$$
This new subsequence satisfies all the conditions of  Lemma \ref{lem:uv}, and so we have:
\begin{equation}\label{eq:L2L3} 
 u^{\eps_k'}_x \phi \to v'\phi \qquad \mbox{ strongly  in } L^2(\Omega) \mbox{ and } L^3(\Omega) .
 \end{equation}

\medskip

Finally, we write (see \eqref{eq:vphiT}): 
\begin{equation}\label{eq:TTE}
\langle T^\eps(t) , \phi^4 \rangle  =  \int_\Omega u^\eps (u^\eps_{xx})^2 \phi^4\, dx - \frac 5 {6}\int_\Omega {u_x^\eps}^3  \phi ^3 4 \phi' \, dx - \frac 1 2 \int_\Omega u^\eps {u^\eps_x}^2  \phi^2 4 [3{\phi'}^2+ \phi \phi'']\, dx.
\end{equation}
Clearly, \eqref{eq:L2L3} allows us to pass to  the limit in the last two terms. So the main difficulty comes from the first term, for which we only have an inequality. 
Indeed,
\eqref{eq:CCF} implies that the function $g^{\eps_k'} (x) = \sqrt{u^{\eps_k'}} (t)u^{\eps_k'}_{xx} (t)\phi^2$ is bounded in $L^2$ and thus converges (up to another subsequence) to $g(x)$ weakly in $L^2$ and the lower semicontinuity of the $L^2$ norm implies
$$
\liminf_{k\to\infty}
 \int_\Omega u^{\eps_k'} (u^{\eps_k'}_{xx})^2 \phi^4\, dx \geq \int_\Omega g^2\, dx
$$
Furthermore, since
 $v$ is in $C^{1/2}(\Omega)$, using the uniform convergence of $u^{\eps_k'}$ and the bound \eqref{eq:bd3dd}, it is easy to show that for all $\delta>0$ the function
$u^{\eps_k'}_{xx}(t,\cdot)$ converges to  $v''$ weakly in $L^2(\{v>\delta\})$. We  deduce that $g=\sqrt v v''$ in $\{v>0\}$ and it follows that
\begin{equation}
\label{eq:gg}
\liminf_{k\to\infty}
 \int_\Omega u^{\eps_k'} (u^{\eps_k'}_{xx})^2 \phi^4\, dx \geq\int_{\{v>0\}} v (v'')^2\phi^4 \, dx
 \end{equation}
 
 We can now pass to the limit in \eqref{eq:TTE} to get:
\begin{align} 
\lim_{k\to\infty} \langle T^{\eps_k'}(t) , \phi^4 \rangle  
& \geq 
\int_{\{v>0\}} v (v'')^2\phi^4 \, dx - \frac 5 {6}\int_{\{v>0\}} {v'}^3 [\phi ^4]'   - \frac 1 2 \int_{\{v>0\}} v {v'}^2 [\phi^4]''\nonumber \\
& \geq \int_{\{v>0\}} v (v'')^2\phi^4 \, dx - \frac 1 3\int_{\{v>0\}} {v'}^3 [\phi^4]' +  \int_{\{v>0\}} v v' v'' [\phi^4]' .\label{eq:TTFT}
\end{align}
Finally, we recall (see Proposition \ref{prop:uacc}) that
$$ v''' = 0 \mbox{ in } \{v>0\}$$
from which we deduce (after a few straightforward integration by parts in \eqref{eq:TTFT}):
\begin{equation}\label{eq:T}
\liminf_{k\to\infty} \langle T^{\eps_k}(t) , \phi^4 \rangle   \geq \frac 1 3  \int_{\pa\{v>0\}} |v'|^3 \phi^4 d\mathcal H^0(x) \geq 0
\end{equation}
which completes the proof.
\end{proof}

Note that we have in fact proved the following stronger result, which will be useful later on:
\begin{proposition}\label{prop:1}
For a given $t\in\P$,
let $\eps_k'$ be any subsequence of $\eps_k$ such that $u^{\eps_k'}(t)$ converges uniformly and in $H^1$ weak to a function $v(x)$.
Then for positive any test function $\phi\in \mathcal{D}(\overline{\Omega})$, we have
$$ \liminf_{k\to\infty} \langle T^{\eps_k'}(t) , \phi^4 \rangle   \geq
\frac 1 3   \int_{\pa\{v>0\}} |v'(x)| ^3 \phi (x)^4 d \mathcal H^0(x).$$
\end{proposition}

We end this section with the proof of the key lemmas:

\begin{proof}[Proof of Lemma \ref{lem:matias}]
Using \eqref{eq:vphiT}, we can write:
\begin{equation}\label{eq:vphiT2} 
\langle T^\eps , \phi^4\rangle  =  \int_\Omega u^\eps (u^\eps_{xx})^2 \phi^4 \, dx - \frac {10} {3}\int_\Omega {u_x^\eps}^3 \phi^3 \phi' \, dx- \frac 1 2 \int_\Omega u^\eps {u^\eps_x}^2 \phi^2 [ 4\phi \phi'' + 12{\phi'}^2]\, dx.
\end{equation}

Next, we show that for any function $v(x)$ such that $vv'=0$ on $\pa\Omega$, we have the inequality:
\begin{equation} \label{eq:v4}
 \int_\Omega |v'|^4 \phi^4 \, dx\leq C \| v\|_{L^\infty(\supp\phi) } \int_\Omega v {v''}^2 \phi^4 \, dx
 + C \| \phi'\|_{L^\infty}^4 \int_{\supp\phi} v^4 \, dx.
 \end{equation}
Indeed, using an integration by parts and the fact that $v v_x = 0 $ on $\pa\Omega$, we deduce
\begin{align*}
 \int_\Omega |v'|^4 \phi^4 \, dx 
 & = -3\int_\Omega v {v'}^2 v''\phi^4 \, dx - 4 \int_\Omega v {v'}^3 \phi^3 \phi'\\
 & \leq 3 \| v\|_{L^\infty}^{1/2} \left(\int _\Omega v ( v'')^2 \phi^4 \, dx\right)^{1/2} \left( \int _\Omega v'^4 \phi^4\, dx \right)^{1/2} \\
 & \qquad + 4 \left( \int_\Omega {v'}^4 \phi^4\, dx \right)^{3/4} \left( \int_\Omega v^4 {\phi'}^4\, dx\right)^{1/4}
 \end{align*}
 which implies \eqref{eq:v4}.
 
Using \eqref{eq:v4}, we then get (for any $\lambda>0$):
\begin{align*}
\int_\Omega {u_x^\eps}^3 \phi^3 \phi' \, dx
& \leq 
\left( \int_\Omega {u_x^\eps}^4 \phi^4  \, dx\right)^{3/4}
\left(\int_\Omega {\phi'}^4 \, dx\right) ^{1/4}\\
& \leq \lambda^{4/3} \int_\Omega {u_x^\eps}^4 \phi^4  \, dx + \frac{1}{\lambda^4} \int_\Omega {\phi'}^4 \, dx\\
& \leq
 \lambda^{4/3} C \| u^\eps\|_{L^\infty } \int u^\eps {u^\eps_{xx}}^2 \phi^4 \, dx + \lambda^{4/3} C
 + \frac{1}{\lambda^4} \int_\Omega {\phi'}^4 \, dx 
\end{align*}
for some constant $C$ depending on $\phi$ and $\| u^\eps\|_{L^\infty}$.
Choosing $\lambda^{4/3} = \frac{1}{2C \| u^\eps\|_{L^\infty }}$, we deduce
\begin{equation}\label{eq:v3}
\left| \int_\Omega {u_x^\eps}^3 \phi^3 \phi' \, dx\right|
\leq \frac1  2 \int_\Omega u^\eps {u^\eps_{xx}}^2 \phi^4 \, dx + C.
\end{equation}
We also have
\begin{equation}\label{eq:at}
\left|\int_\Omega u^\eps {u^\eps_x}^2 \phi^2 [ 4\phi \phi'' + 12{\phi'}^2]\, dx\right|
\leq 
C(\phi) \| u\|_{L^\infty}\int_\Omega {u^\eps_x}^2\, dx 
\end{equation}

Equation \eqref{eq:vphiT2} together with  \eqref{eq:v3} and \eqref{eq:at} and the monotonicity of the $H^1$ seminorm \eqref{eq:energy} gives the result.
\end{proof}

\begin{proof}[Proof of Lemma \ref{lem:uv}]
First, we recall that $u^\eps(t,\cdot)$ is bounded in $H^1(\Omega)$. Hence, $u^{\eps_k}_x\phi$ converges to $v_x\phi$ weakly in $L^2(\Omega)$. Also, the assumptions imply that the function $g^{\eps_k} = \sqrt{u^{\eps_k}} u^{\eps_k}_{xx} \phi^2$ is bounded in $L^2(\Omega)$ and thus converges (up to another subsequence) to $g(x)$ weakly in $L^2(\Omega)$.
We now write
\begin{align*}
\int_\Omega |u^{\eps_k}_x \phi |^2\, dx & = -\int_\Omega u^{\eps_k} u^{\eps_k}_{xx}\phi^2\, dx - \int_\Omega u^{\eps_k}u^{\eps_k}_x 2 \phi \phi'\, dx \\
& = -\int_\Omega \sqrt {u^{\eps_k}} g ^{\eps_k} \, dx - \int_\Omega u^{\eps_k}u^{\eps_k}_x 2 \phi \phi'\, dx 
\end{align*}
The convergence above and the uniform convergence of $u^{\eps_k}$ thus imply
$$
\lim_{k\to \infty}
\int_\Omega |u^{\eps_k}_x \phi |^2\, dx = 
 -\int_\Omega \sqrt {v} g   \, dx - \int_\Omega v v' 2 \phi \phi'\, dx 
$$
It remains to see (proceeding as in \eqref{eq:gg}) that 
\begin{equation*}
\sqrt {v}  g = v v'' \mbox{ a.e. in $\Omega$}
\end{equation*}
to conclude that
$$
\lim_{k\to \infty}
\int_\Omega |u^{\eps_k}_x \phi |^2\, dx = \int_\Omega |v' \phi|^2\, dx$$
This implies that $u^{\eps_k}_x \phi $ converges strongly to $v' \phi$ in $L^2$.
We note that the hypothesis \eqref{eq:bd3dd} and the bound \eqref{eq:v4}
imply that $u^\eps_x \phi$ is bounded in $L^4$. A classical argument now yields the strong convergence in $L^p$ for all $p<4$.

We finally note that the limit is independent of any subsequence we have extracted, hence the full sequence $u^{\eps_k}_x \phi$ converges to $v_x\phi$ 
\end{proof}

\section{Proof of Theorem \ref{thm:u}}\label{sec:thmu}

Parts of Theorem~\ref{thm:u} were already proven in Proposition \ref{prop:uacc}:
Up to another subsequence $\eps_k'$, every accumulation point $v(x)$ of $\{u^{\eps_k'}(t,\cdot)\}$ for $t\in \P$  satisfies
$$ v'''=0 \mbox{ in } \{v>0\}, \quad\int_\Omega v(x)\, dx=\int u_{in}(x)\, dx .$$

The only statement that remains to be proved is thus  \eqref{eq:vsupp}.
The first inclusion in  \eqref{eq:vsupp} is a consequence of the uniform converge of $u^{\eps_k'}$ to $v$ and of the properties of the function $\rho$, as shown in the following lemma:
\begin{lemma}
There exists a subsequence $\eps_k'$ of $\eps_k$ such that 
for almost every $t>0$ and for every accumulation point $v$ of $\{u^{\eps_k'}(t,\cdot)\}$, we have
$$\{v>0\}\subset \Sigma(t)$$.
\end{lemma}
\begin{proof}
We recall that by Theorem~\ref{thm:rho}, we have 
\begin{equation*} 
 \rho^{\eps_k} (\cdot,t)\to \rho(\cdot,t) \mbox{ strongly in $W^{-1,1}(\Omega)$  for all  $t\in \P$. }
\end{equation*}

Next, we recall that $H^1(\Omega) \subset C^{1/2}(\Omega)$ and so the functions $x\mapsto u^{\eps_k} (t,x)$ are equi-continuous with respect to $k$.
Let now $v(x)$ be such that
$$ u^{\eps_k'}(t,x) \to v(x) \mbox{ uniformly in $x$ and weakly in $H^1(\Omega)$}.$$

If $x_0\in \{v>0\}$, then there exists $\delta>0$ and $r>0$ such that
$$ u^{\eps_k'}(t,x)\geq \delta \mbox{ in $B_r(x_0)$ and for all $k$ large enough}.$$
In particular, we deduce
$$ \rho^{\eps_k'} (t,x)\geq B^{\eps_k'} (\delta) 
\mbox{ in $B_r(x_0)$ and for all $k$ large enough}.$$
Passing to the limit, we deduce that $\rho(t,x)=1$ in $B_r(x_0)$ and the definition of $\Sigma(t)$ implies that $x_0\in \Sigma(t)$.
\end{proof}

The second inclusion in \eqref{eq:vsupp} is more delicate. It comes from the fact that on the set $\Sigma(t)$, $\rho$ is maximum and so $\pa_t \rho \leq 0$. Therefore, $\pa_t \rho$ must vanish on that set, since $t\mapsto \rho(t,x)$ is non-decreasing and this implies that the distribution $T$ should be zero there. 
If there was a point $x_0\in \pa\{v>0\}\cap \Sigma(t)$, this would then contradict inequality \eqref{eq:T}.

However, because $\pa_t \rho^{\eps_k}$ does not converge pointwise, this will only hold 
after extracting another subsequence, as stated in the following result:
\begin{proposition}\label{prop:suppu}
There exists a subsequence $\eps_k'$ of $\eps_k$ such that 
for every $t\in \P$ and for every accumulation point $v$ of $\{u^{\eps_k'}(t,\cdot)\}$, we have
$$ \pa \{v>0\}\cap \Sigma(t) = \emptyset.$$
\end{proposition}

\begin{proof}[Proof of Proposition \ref{prop:suppu}]
First, we note that given a test function $\phi$ such that $\supp \phi \subset \Sigma(s)$, then for all $t>s$ we have 
\begin{align*}
\lim_{k\to\infty} \int_s^t  \langle T^{\eps_k}(\tau), \phi^4\rangle \, d\tau & =  
\lim_{k\to \infty}\left( \int_\Omega \rho^{\eps_k}(t,x) \phi(x)^4\, dx - \int \rho^{\eps_k}(s,x) \phi(x)^4\, dx\right)\\
& = \int_\Omega \rho(t,x)\phi(x)^4\, dx -  \int_\Omega \phi(x)^4\, dx\\
& \leq 0.
\end{align*}
Using Proposition \ref{prop:T2}, we deduce that
$$ \lim_{k\to\infty} \int_s^t  \langle T^{\eps_k}(\tau), \phi^4\rangle \, d\tau = 0 $$
and 
$$\liminf _{k\to\infty}\langle T^{\eps_k}(t), \phi^4\rangle  = 0 \qquad \mbox{ a.a. $t>s$}.$$

This implies that the function $\langle T^{\eps_k}(\tau), \phi^4\rangle  -\inf _{k'\geq k}\langle T^{\eps_k'}(t), \phi^4\rangle $, which is non-negative and bounded in $L^1(0,T)$ by Lemma~\ref{lem:matias}, satisfies 
$$ \int_s^t  \langle T^{\eps_k}(\tau), \phi^4\rangle  -\inf _{k'\geq k}\langle T^{\eps_k'}(t), \phi^4\rangle   \, d\tau \to 0$$
and thus
converges to zero in $L^1(s,t)$.
We deduce that up to another subsequence, we have 
$$\lim _{k\to\infty}\langle T^{\eps_k'}(t), \phi^4\rangle  = 0 \qquad \mbox{ a.a. $t>s$}.$$

Of course, this subsequence depends  on $s$ and on the test function $\phi$.
However, if we consider a dense subset $\{s_l\}_{l\in \NN} \subset (0,\infty)$
and a countable family of smooth function $\phi_{l,m}$ such that for all $l\in \NN$ we have
$$\supp \phi_{l,m}\subset \Sigma(s_l) \mbox{ for all $m\in\NN$}, \qquad \lim_{m\to \infty} \phi_{l,m} = \chi_{\Sigma(s_l)},$$
we can use a diagonal extraction argument to extract a subsequence $\eps_k'$ such that
for all $(l,m)\in \NN\times \NN$ 
$$\lim _{k\to\infty}\langle T^{\eps_k'}(t), \phi_{l,m}^4\rangle  = 0 \qquad \mbox{ a.a. $t>s_l$}.$$

Since there are countably many $s_l$, we deduce that for a.e. $t>0$ and for all $n$ such that $s_l<t$, we have
$$\lim _{k\to\infty}\langle T^{\eps_k'}(t), \phi_{l,m}^4\rangle  = 0 \qquad \forall m.$$

Using Proposition \ref{prop:1}, we deduce that  for almost every  $t>0$, if $v$ is an accumulation point of the subsequence  $u^{\eps_k'}(t)$ defined above, 
then
$$\frac 1 3   \int_{\pa\{v>0\}} |v'(x)| ^3 \phi_{l,m} (x)^4 d \mathcal H^0(x) = 0 $$
for all $l$ such that $s_l<t$ and all $m\in \NN$.
Since we know that $v'\neq 0$ on $\pa\{v>0\}$ (recall that $v$ is a parabola), we deduce that 
$$\pa\{v>0\} \cap \Sigma(s_l) = \emptyset \qquad\mbox{for all $l$ such that $s_l<t$.}$$

To conclude, we show that 
\begin{equation}\label{eq:sigma t equals limit}
\Sigma(t) = \bigcup_{l\mbox{ s.t. } s_l<t} \Sigma(s_l) \mbox{ for a.e. $t>0$}.    
\end{equation}
This follows from the monotonicity of $\rho$ \eqref{eq:positivity}, which implies that the sets $\Sigma(t)$ are increasing in $t$. Therefore, the quantity $|\Sigma(t)|$ is monotonic in time and can only have a countable number of discontinuities. In particular, for every $t$ which is a continuity point of $|\Sigma(\cdot)|$, we have
$$
\left|\Sigma(t)\setminus \bigcup_{l \mbox{ s.t. } s_l<t} \Sigma(s_l)\right|=|\Sigma(t)|-\lim_{s\to t^-}|\Sigma(s)|=0.
$$
This implies \eqref{eq:sigma t equals limit}.
\end{proof}

\section{Proof of Theorem \ref{thm:sup}}\label{sec:thmsup}
The goal of this section is to establish the inequality \eqref{eq:tannerglobal}.
In fact, we will prove the following slightly stronger statement:
\begin{proposition}\label{prop:sup1}
For a given test function $\phi\in \mathcal D(\overline \Omega)$, there exists a function $w(t,x)$ defined for $t\in \P$ and $x\in \Omega$ such that the following inequality holds in the sense of distribution:
$$ 
\frac{d}{dt} \int_\Omega \rho(t,x) \phi^4(x)\, dx \geq 
 \frac 1 3  \int_{\pa\{w>0\}} |w_x(t,x)| ^3 \phi(x)^4 d \mathcal H^0(x) 
$$
where 
\begin{equation}\label{eq:wcond}
    w_{xxx} = 0 \mbox{ in } \{w>0\} \subset \Sigma(t),\quad \pa\{w>0\}\cap \Sigma(t)=\emptyset, \quad \int_\Omega w(t,x)\, dx =1.
\end{equation}
\end{proposition}

Applying this result to the function $\phi=1$ gives \eqref{eq:tannerglobal}.
\medskip

We first prove:
\begin{proposition}\label{prop:2}
For a given $t\in\P$ and for a given test function $\phi\in \mathcal D(\overline \Omega)$, there exists an accumulation point $v$ of $\{u^{\eps_k}(t)\}_{k\in \NN}$ such that
$$ \liminf_{k\to\infty} \langle T^{\eps_k}(t) , \phi^4 \rangle   \geq
\frac 1 3  \int_{\pa\{v>0\}} |v'(x)| ^3 \phi(x)^4 d \mathcal H^0(x) 
$$
\end{proposition}
\begin{proof}
We take a subsequence $\eps_k'$ such that
$$ \lim_{k\to\infty} \langle T^{\eps_k'}(t) , \phi^4 \rangle  = \liminf_{k\to\infty} \langle T^{\eps_k}(t) , \phi^4 \rangle. $$
Up to a further subsequence, there exists $v$ such that
$$
u^{\eps_k'}(t)\to v
$$
uniformly and weakly in $H^1$. Then Proposition \ref{prop:1} implies the result.
\end{proof}
\begin{proof}[Proof of Proposition~\ref{prop:sup1}]
For each $t\in \mathcal{P}$, we define $w(t,x)=v(x)$, where $v$ is given by Proposition~\ref{prop:2}, where we have used the sequence $\eps'_k\to 0$ given by Theorem~\ref{thm:u}. It follows by Theorem~\ref{thm:u} that $w$ satisfies \eqref{eq:wcond}.

We take a positive test function $\eta\in\mathcal{D}([0,T])$, and we notice that because the limit exists
$$
\lim_{k\to\infty}\int_0^T \langle \pa_t\rho^{\eps_k}(t),\phi^4\rangle \eta(t)\, dt=\lim_{k\to\infty}\int_0^T \langle \pa_t\rho^{\eps'_k}(t),\phi^4\rangle \eta(t)\, dt,
$$
where $\eps'_k$ is the subsequence of $\eps_k$ given by Theorem~\ref{thm:u}. We forego relabeling the subsequence for notational convenience.

Using our decomposition of the time derivative of $\rho^\eps$ \eqref{eq:rhoeps1}, we have
$$
\int_0^T \langle \pa_t\rho^{\eps_k},\phi^4\rangle \eta\, dt= \int_0^T \langle \pa_x R^{\eps_k},\phi^4\rangle \eta\, dt+\int_0^T \langle T^{\eps_k},\phi^4\rangle \eta\, dt.
$$
By Lemma~\ref{lem:Reps00}, we know that the first term on the right hand side vanishes in the limit. By the lower bound \eqref{eq:matias}, we can apply Fatou's Lemma to obtain that
$$
\lim_{k\to\infty} \int_0^T \langle T^{\eps_k},\phi^4\rangle \eta\, dt\ge \int_0^T \liminf_{k\to\infty}  \langle T^{\eps_k},\phi^4\rangle \eta\, dt\ge \frac 1 3  \int_0^T\int_{\pa\{w>0\}} |w_x| ^3 \phi^4\eta\, d \mathcal H^0dt,
$$
where in the last inequality we have used our definition of $w$ and Proposition~\ref{prop:2}.

\end{proof}

\section{Proof of Theorem \ref{thm:cond1}}\label{sec:cond}
In this section, we prove the conditional result Theorem \ref{thm:cond1}.
First, we state the following lemma, which we will prove at the end of this section:
\begin{lemma}\label{lem:cond}
Assume that the sequence $u^{\eps_k}(t,x)$ is such that for almost every $t\in (0,T)$ we have
$$ u^{\eps_k}(t) \to w(t) \mbox{ uniformly and weakly in $H^1(\Omega)$}$$
and 
\begin{equation}\label{eq:cond2}
\lim_{k\to \infty} \int_0^T \int_\Omega u^{\eps_k} (u^{\eps_k}_{xx})^2\, dx\, dt 
= \int_0^T \int_{\{w>0\}} w (w_{xx})^2 \, dx\, dt.
\end{equation}
Then, 
\begin{equation}\label{eq:L1convergence}
\int_\Omega u^{\eps_k} (u^{\eps_k}_{xx})^2\, dx\to \int_{\{w>0\}} w (w_{xx})^2 \, dx\qquad\mbox{in $L^1(0,T)$}    
\end{equation}
and, up to another subsequence, we can assume that for a.a. $t\in(0,T)$
and for all test function  $\vphi\in \mathcal D(\overline \Omega)$
 we have
$$
\lim_{k\to \infty}  \int_\Omega u^{\eps_k} (u^{\eps_k}_{xx})^2 \vphi(x)^4\, dx
=\int_{\{w>0\}} w (w_{xx})^2 \vphi(x)^4\, dx 
$$
and 
$$
u^{\eps_k}_x(t) \vphi \to w_x(t)\vphi \qquad\mbox{strongly in $L^p (\Omega)$ for all $p<4$.}
$$
\end{lemma}

\begin{proof}[Proof of Theorem \ref{thm:cond1}]
We fix a test function $\vphi$ and  consider a subsequence (still denoted $\eps_k$) such that the result of Lemma \ref{lem:cond} holds.
We recall (see \eqref{eq:TTE}) that
$$
\langle T^\eps(t) , \vphi^4 \rangle  =  \int_\Omega u^\eps (u^\eps_{xx})^2 \vphi^4\, dx - \frac 5 {6}\int_\Omega {u_x^\eps}^3 ( \vphi^4)' \, dx - \frac 1 2 \int_\Omega u^\eps {u^\eps_x}^2 (\vphi^4)''\, dx.
$$
Using Lemma \ref{lem:cond} and condition \eqref{eq:cond0} we can pass to the limit in this equality and get (for a.e. $t\in(0,T)$):
$$
\begin{array}{rcl}
\displaystyle\lim_{k\to\infty}\langle T^{\eps_k}(t), \vphi^4\rangle&=&\displaystyle \int_{\{w>0\}} w(t) (w_{xx}(t))^2 \vphi^4 \, dx - \frac 5 {6}\int_{\{w>0\}} {w_x}^3(t) (\vphi^4)' \, dx- \frac 1 2 \int_{\{w>0\}} w(t) {w_x}^2(t) (\vphi^4)''\, dx\\
& = &\displaystyle\frac{1}{3}\int_{\pa\{w(t)>0\}}|w_x(t)|^3\vphi(x)\,d\mathcal{H}^0.
\end{array}
$$
Furthermore, using  \eqref{eq:v3} and \eqref{eq:at}, we easily get:
\begin{align*}
 &\left| - \frac 5 {6}\int_\Omega {u_x^{\eps_k}}^3 ( \vphi^4)' \, dx - \frac 1 2 \int_\Omega u^{\eps_k} {u^{\eps_k}_x}^2 (\vphi^4)''\, dx\right| \\
&\qquad\qquad\qquad\qquad  = \left|  
- \frac {10} {3}\int_\Omega {u_x^{\eps_k}}^3 \vphi'\vphi^3 \, dx- \frac 1 2 \int_\Omega u^{\eps_k} {u^{\eps_k}_x}^2 \vphi^2[4\vphi\vphi''+12\vphi'^2]\, dx\right|\\
& \qquad\qquad\qquad\qquad\le C(\vphi)\left(\int_\Omega u^{\eps_k} (u^{\eps_k}_{xx})^2 \, dx+1\right), 
\end{align*}
which implies
$$
|\langle T^{\eps_k}(t), \vphi^4\rangle|\le C(\vphi)\left(\int_\Omega u^{\eps_k} (u^{\eps_k}_{xx})^2 \, dx+1\right).
$$
Since the right hand side converges in $L^1(0,T)$ (by the hypothesis \eqref{eq:cond0}),  we can use Lebesgue dominate convergence to show that for any test function $h\in \mathcal D([0,T))$
$$ \lim_{k\to\infty}
 \int_0^{T} \langle T^{\eps_k}(t), \vphi^4\rangle h(t) \,dt
= 
\frac{1}{3}\int_0^T h(t)\left(\int_{\pa\{w(t)>0\}}|w_x(t)|^3\vphi(x)\,d\mathcal{H}^0(x)\right)\,dt,
$$

Finally, 
for every test functions $\vphi \in \mathcal D( \overline \Omega)$ and $h\in \mathcal D([0,T))$
equation \eqref{eq:tfed} gives (using \eqref{eq:vphiT})
\begin{align*}
 -\int_0^T \int_\Omega \rho^{\eps_k}(t,x)h'(t) \vphi(x)^4 \, dx  \, dt&  = \int_\Omega \rho^{\eps_k}_{in}(x)\vphi(0,x) \, dx  - \int_0^T \int_\Omega R^{\eps_k} (t,x) h (t)\vphi(x)^4  \, dx\, dt\nonumber \\
&  \qquad + \int_0^{T} \langle T^{\eps_k}(t), \vphi^4\rangle h(t) \,dt
\end{align*}
and passing to the limit along an appropriate subsequence, we deduce
$$
-\int_0^T \int_\Omega \rho (t,x)h'(t) \vphi(x)^4 \, dx  \, dt = \int_\Omega \rho_{in} \vphi(0,x) \, dx
+\frac{1}{3}\int_0^T h(t)\left(\int_{\pa\{w(t)>0\}}|w_x(t)|^3\vphi(x)\,d\mathcal{H}^0(x)\right)\,dt
$$
which is the weak formulation of \eqref{eq:besteq}.
Since this equation shows that for all $t>0$, $\pa\rho(t,\cdot)$ is a measure supported on $\pa\{w(t)>0\} = \pa \Sigma(t)$, the monotonicity of $\Sigma$ implies that $\rho = 0$ in $\Sigma(t)^c$ and so 
$\rho(t) = \chi_{\Sigma(t)}$.
\end{proof}

\begin{proof}[Proof of Lemma \ref{lem:cond}]
We define $\mu^k = u^{\eps_k} (u^{\eps_k}_{xx})^2$ which is bounded in $L^1((0,T)\times\Omega)$ and therefore converges (up to a subsequence) weakly to a measure $\mu \in \mathcal M((0,T)\times\Omega)$.
For every test function $\psi(t,x)$, we can prove (see \eqref{eq:gg}) that for almost every $t\in(0,T)$,
\begin{equation}\label{eq:muk}
\liminf_{k\to\infty} \int_\Omega \mu^k(t,x) \psi(t,x)\, dx \geq \int_{\{w>0\}} w (w_{xx})^2 \psi(t,x) \, dx.
\end{equation}
Indeed, for all $\eta>0$, we have 
$$ \liminf_{k\to\infty}  \int_\Omega \mu^k(t,x) \psi(t,x)\, dx\geq 
\liminf_{k\to\infty}  \int_{\{w>\eta\}} \mu^k(t,x) \psi(t,x)\, dx
=  \int_{\{w>\eta\}} w (w_{xx})^2 \psi(t,x) \, dx
$$
where the last equality follows from the uniform convergence of $u^{\eps_k}$ and \eqref{eq:limh} 
which provides the strong convergence of $u^{\eps_k}_{xx}$ to $w_{xx}$ in the set $\{w>\eta\}$.

Using Fatou's lemma, \eqref{eq:muk} implies:
$$
\liminf_{k\to\infty} \int_0^T\int_\Omega \mu^k(t,x) \psi(t,x)\, dx \, dt\geq \int_0^T\int_{\{w>0\}} w (w_{xx})^2 \psi(t,x) \, dx\, dt$$
Since the sequence on the left converges to $\int\int \mu \psi$, we deduce that
$ \mu -  w (w_{xx})^2\chi_{\{w>0\}} $ is a non-negative measure.
Assumption \eqref{eq:cond2} says that this measure has zero mass, so we deduce
\begin{equation}\label{eq:mu} \mu =   w (w_{xx})^2\chi_{\{w>0\}} .
\end{equation}

We start by showing $L^1$ convergence in time \eqref{eq:L1convergence}. To simplify, we denote
$$f_k(t):=\int_\Omega u^{\eps_k} (u^{\eps_k}_{xx})^2 \, dx, \qquad\mbox{and}\qquad f(t) :=  \int_{\{w>0\}} w (w_{xx})^2 \, dx.$$
We know (see \eqref{eq:muk}) that
$$\liminf_{k\to\infty} f_k(t) \geq f(t)\qquad \forall t\in(0,T)$$
and so (using Fatou)
$$ \liminf_{k\to \infty} \int_0^T f_k(t)\, dt \geq \int_0^T \liminf_{k\to\infty} f_k(t)\, dt \geq \int_0^T f(t)\, dt.$$
But since  condition \eqref{eq:mu} implies 
$$ \lim_{k\to \infty} \int_0^T f_k(t)\, dt = \int_0^T \int_\Omega \mu \, dx dt= \int_0^T f(t)\, dt,$$
we must have equality in those inequalities, that is
$$ \liminf_{k\to\infty} f_k(t) = f(t) \quad \mbox{ a.e. }  t\in(0,T).$$
In particular the sequence $\inf_{l\geq k} f_l(t)$ is monotone increasing and converges to $f(t)$ almost everywhere. We deduce (by Beppo-Levy monotone convergence theorem) that 
$$ \lim_{k\to\infty} \int_0^T \inf_{l\geq k} f_l(t)\, dt =  \int_0^T f(t)\, dt$$
and so (using \eqref{eq:cond2} again)
$$  \lim_{k\to\infty} \int_0^T f_k(t)-  \inf_{l\geq k} f_l(t)\, dt = 0.$$
Therefore, by the triangle inequality and the two previous equations we obtain the $L^1$ convergence \eqref{eq:L1convergence}
$$
\lim_{k\to\infty}\int_0^T |f_k(t)-f(t)|\,dt\le \lim_{k\to\infty}\int_0^T |f_k(t)-\inf_{l\geq k} f_l(t)|\,dt+\lim_{k\to\infty}\int_0^T |f(t)-\inf_{l\geq k} f_l(t)|\,dt=0.
$$
Hence, there is a subsequence such that
$$  \lim_{k'\to\infty}  f_{k'}(t) = f(t)  \quad \mbox{ for a.e. }  t\in(0,T).$$
In particular, for a.e. $t\in(0,T)$
\begin{equation}\label{eq:bound}
    \sup_{k'}\int u^{\eps_{k'}}(t)|u^{\eps_{k'}}_{xx}(t)|^2\,dt<C.
\end{equation}
Up to a further subsequence, there exists a positive measure $\mu(t)$ such that
$$
u^{\eps_{k''}}(t)|u^{\eps_{k''}}_{xx}(t)|^2\rightharpoonup\mu(t).
$$
Using the same arguments as above, it follows that for a.e. $t\in(0,T)$
$$
\mu(t)=w(t) |w_{xx}(t)|^2 \chi_{\{w>0\}}.
$$
This implies that for every test function $\vphi\in \mathcal D(\Omega)$, we have the desired convergence
\begin{equation*}
\lim_{k'\to \infty }
 \int_\Omega u^{\eps_k'} (u^{\eps_k'}_{xx})^2 \vphi^4  \, dx
= \int_{\{w>0\}} w (w_{xx})^2 \vphi^4  \, dx \quad\mbox{ a.e. $t\in(0,T)$.}
\end{equation*}

The last part of the Lemma follows from the bound \eqref{eq:bound} and Lemma~\ref{lem:uv}.
\end{proof}


\section{Proof of  Corollary \ref{cor:ODE for the integral}}\label{sec:cor}
\begin{proof}[Proof of  Corollary \ref{cor:ODE for the integral}]
We start by writing 
$$\Sigma(t)=\bigcup_{i= 1}^{N_0} (a_i,b_i)$$ 
with $N_0\in \NN\cup{+\infty}$ and $a_i<b_i<a_{i+1}$ for every $i$. 
According to Theorem~\ref{thm:u}, if $a_1\in\Omega$ and $b_{N_0}\in\Omega$ i.e. $\pa\Sigma(t)\subset \Omega$, then there exists a family of positive numbers $\{\gamma_i\}_{i\le N_0}$ such that $\sum_{i=1}^{N_0}\gamma_i=1$ and
$$
w(t,x)=
6\sum_{i\in I} \gamma_i \frac{(b_i-x)_+(x-a_i)_+}{(b_i-a_i)^3}.
$$
This explicit expression for $w$ then implies the following chain of inequalities
$$
\begin{array}{rcl}
\displaystyle\frac 1 3  \int_{\pa\{w>0\}} |w_x(t,x)| ^3  d \mathcal H^0(x)&\geq&\displaystyle 144\sum_{i=1}^{N_0}\gamma_i(b_i-a_i)^{-6}\\
&\geq&\displaystyle 144 |\Sigma(t)|^{-6}\\
 &\geq&\displaystyle 144\left(\int_\Omega \rho(t,x)\,dx\right)^{-6}.
\end{array}
$$
By the monotonicity of $\rho$ \eqref{eq:positivity}, we have that if $\pa\Sigma(t_0)\subset \Omega$, then $\pa\Sigma(t)\subset \Omega$ for every $t<t_0$.

On the other hand, if $a_1(t)\in\pa\Omega$ and $\Sigma(t)\ne\Omega$, then we have the alternative
$$
w(t,x)\chi_{(a_1,b_1)}=\gamma_1 6 \frac{(b_1-x)(x-a_1)}{(b_1-a_1)^3}
$$
or 
$$
w(t,x)\chi_{(a_1,b_1)}=\gamma_1 \frac{3}{2}\frac{(b_1-a_1)^2-(x-a_1)^2}{(b_1-a_1)^3},
$$
where we have used the boundary condition for $u^\eps_x(t,a_1)=0$ \eqref{eq:bc}, which is preserved if $w(t,a_1)>0$. Calculating the derivative explicitly and bounding, we have the inequalities
$$
\begin{array}{rcl}
\displaystyle\frac 1 3  \int_{\pa\{w>0\}} |w_x(t,x)| ^3  d \mathcal H^0(x)&\geq&\displaystyle 9\sum_{i=1}^{N_0}\gamma_i(b_i-a_i)^{-6}\\
&\geq&\displaystyle9|\Sigma(t)|^{-6}\\
 &\geq&\displaystyle 9\left(\int_\Omega \rho(t,x)\,dx\right)^{-6}.
\end{array}
$$
The result follows from the previous inequalities, the conclusion of Theorem~\ref{thm:sup} \eqref{eq:tannerglobal} and integrating.

\end{proof}

\appendix
\section{Existence: Proof of Proposition~\ref{prop:exist}}\label{ap:exist}
The existence of solutions for the  thin film equation in dimension one and when $n\in(0,3)$ is classical (see  in particular \cite{BF90,BP96}).
We present a proof of Proposition~\ref{prop:exist} here because some of the properties that we need for our analysis (in particular the equation for $\rho^\eps$), while not difficult to establish, are not included in the typical results found in the litterature.
However,
for simplicity, we will only write the complete proof in the case $n=1$; the case $n\in(1,2)$ is completely analogous and the case $n\in(0,1)\cup [2,3)$ requires a bit more care (in particular one has to regularize the initial data). We refer the interested reader to \cite{BP96} for further details.

\begin{proof}[Proof of Proposition~\ref{prop:exist}]
Since $\eps>0$ is fixed, we prove the existence of a solution to the  rescaled equation (see remark \ref{rem:1}):
\begin{equation}\label{eq:tf0}
\pa_t u + \pa_x (u(1+u^2)\pa_{xxx} u )=0.
\end{equation}
The basic idea for the proof of Proposition~\ref{prop:exist} follows the same
approximation argument found for example in \cite{BF90}. 
However we do need to check carefully that the additional properties that we need to carry out our analysis hold.

\paragraph{Regularization.}
First, we approximate equation \eqref{eq:tf0} by a uniformly parabolic equation:
We set $f(u):=|u|(1+u^2)$ and 
$$
f_\delta(u)=f(u)+\delta.
$$
Then (see \cite{BF90}), the regularized equation
\begin{equation}\label{eq:delta}
\left\{
\begin{array}{ll}
\pa_t u +\pa_x (f_\delta(u)\pa_{xxx}u)=0 & \mbox{ in } \Omega\times (0,\infty) \\[5pt]
f_\delta(u) \pa_{xxx}u = 0,\qquad \pa_x u=0  & \mbox{ on } \pa \Omega \times (0,\infty)\\[5pt]
u(x,0) = u^\delta_{in}(x) & \mbox{ in } \Omega 
\end{array}
\right.
\end{equation}
has a unique classical (smooth) solution $u^\delta(t,x)$ for all $\delta>0$, where $u^\delta_{in}$ is a smooth positive approximation of $u_0$. Furthermore, this solution satisfies the mass conservation equality
\begin{equation}\label{eq:mass0}
\int_\Omega u^\delta(t,x)\, dx = \int_\Omega u^\delta_{in}(x)\, dx
\end{equation}
and the energy equality
\begin{equation}\label{eq:energydelta}
\frac{1}{2}\int_\Omega |\pa_x u^\delta(t,x)|^2\,dx+\int_0^t\int_\Omega f_\delta(u^\delta)|\pa_{xxx}u^\delta(s,x)|^2\,dx\,ds=\frac{1}{2}\int_\Omega |\pa_x u_{in}^\delta(x)|^2\,dx
\end{equation}
(obtained by multiplying the equation  by $-u^\delta_{xx}$).

\paragraph{Uniform convergence of $u^\delta$.}
Equalities \eqref{eq:mass0}
 and \eqref{eq:energydelta} imply that $u^\delta$ is bounded in $L^\infty(0,\infty ; H^1(\Omega))$
 and thus in $L^\infty(0,\infty;C^{1/2}(\Omega))$.
 A classical argument (see \cite{BF90}) then implies that $u^\delta $ is bounded in $C^{1/8}(0,\infty ; C^{1/2}(\Omega))$.
 Up to a subsequence, we can thus assume that 
 $$ u^\delta (t,x)\to u(t,x) \mbox{ uniformly  and in } L^\infty(0,T;H^1(\Omega))\mbox{-weak}.$$

\paragraph{Non-negative solutions.}
Because \eqref{eq:delta} is a fourth order equation, the solution $u^\delta$ may take negative values. In \cite{BF90}, it is shown, that up to a subsequence $u^\delta$ converges uniformly to a non-negative weak solution of the thin film equation $\eqref{eq:tfe}$ satisfying the boundary conditions \eqref{eq:bc} and the conservation of mass \eqref{eq:mass}. 

This argument makes use of another entropy estimate: We introduce the function
$$
H_\delta (s) = \int_s^\infty \int_r^\infty \frac{1}{f_\delta(u)}\, du\, dr$$
A simple computation using the fact that $ H_\delta''(s) = \frac{1}{f_\delta(s)}$ and the boundary conditions implies
\begin{equation}\label{eq:entropyH} 
\int_\Omega H_\delta(u^\delta(t,x))\, dx  +\int_0^t \int_\Omega |u^\delta_{xx}| ^2\, dx\, ds =\int_\Omega H_\delta(u^\delta_{in}(x))\, dx .
\end{equation}
Furthermore, we have 
$$ \lim_{\delta \to 0 } H_\delta (s) =  H_0(s):=\int_s^\infty \int_r^\infty \frac{1}{(1+u^2)|u|}\, du\, dr
=
\left\{
\begin{array}{ll}
  \arctan(\frac 1 s) - \frac 1 2   s   \ln(1+\frac{1}{s^2}) & \mbox{ for } s\geq  0\\
  +\infty & \mbox{ for } s< 0
  \end{array}\right.
  .
$$
Proceeding as in \cite{BF90}, \eqref{eq:entropyH} implies that $u(t,x)=\lim_{\delta\to0} u^\delta(t,x) \geq 0$
and that 
$$
\int_\Omega H_0(u(t,x))\, dx  +\int_0^t \int_\Omega |u_{xx}| ^2\, dx\, ds =\int_\Omega H_0(u_{in}(x))\, dx 
$$
which gives in particular that $u_{xx}\in L^2((0,\infty)\times\Omega)$.

\paragraph{Energy inequality.}
Next, we pass to the limit in \eqref{eq:energydelta}:
The function $g^\delta =\sqrt{f_\delta(u^\delta)}\pa_{xxx}u^\delta$ is bounded in $L^2((0,\infty)\times \Omega)$ and thus weakly converges to a function $g$. 
Using the fact  that $u^\delta$ converges uniformly we can then show that $g= \sqrt{f(u)}\pa_{xxx}u$ in the set $\{u>0\}$. 
Using the lower semicontinuity of the $L^2$ norm we can now pass  to the limit in \eqref{eq:energydelta} and  obtain \eqref{eq:energy}.

The uniform convergence of $u^\delta$ implies  
\begin{equation}\label{eq:fdeltalim}
f_\delta(u^\delta) u_{xxx}^\delta  \rightharpoonup \sqrt{f(u)} g = 
\left\{
\begin{array}{ll} 
 f(u) u _{xxx} & \mbox{  in } \{u>0\}\\
0 &   \mbox{  in } \{u=0\}
\end{array}
\right. \mbox{ weakly in $L^2$.}
\end{equation}
With this, we can easily pass to the limit in the weak formulation of equation \eqref{eq:delta}.

\paragraph{Equation for $\rho(t,x)$ and entropy inequality}
This is the main novelty here, since the function $\rho$ does not appear in  
previous papers on this topic.
We recall that 
$$ \rho(t,x) = B(u(t,x))$$
and we introduce 
$$ \rho^\delta = B_\delta(u^\delta)$$
where 
$ B_\delta (s)$ is defined by
$$ B_\delta(s)= \int_0^s \int_r^\infty \frac{u}{f_\delta(u)}\, du\, dr.$$

We notice that $B^{\prime\prime}_\delta\to_{\delta\to0} B^{\prime\prime}$ in $L^1(\RR)\cap L^\infty(\RR)$, which implies that $B_\delta\to B$ uniformly on bounded sets. Therefore, we have the convergence $\rho^\delta \to \rho$ uniformly.

Next, a simple computation yield 
\begin{equation}\label{eq:tfe5} 
\pa_t \rho^\delta = T^\delta+\pa_x R^\delta
\end{equation}
where
\begin{align*}
T^\delta  &= - u^\delta  u_x^\delta u_{xxx}^\delta\\[5pt]
R^\delta &= - B_{\delta}^{\prime}(u^\delta)f_\delta(u^\delta) u_{xxx}^\delta.
\end{align*}
We can pass to the limit in $R^\delta$ by using \eqref{eq:fdeltalim}.
To pass to the limit in the definition of $T^\delta$, we first notice, 
proceeding as with \eqref{eq:fdeltalim}, that 
$$
u^\delta u_{xxx}^\delta  \rightharpoonup  h  = 
\left\{
\begin{array}{ll} 
u u _{xxx} & \mbox{  in } \{u>0\}\\
0 &   \mbox{  in } \{u=0\}
\end{array}
\right. \mbox{ weakly in $L^2$.}
$$
Furthermore, \eqref{eq:entropyH} implies that $u^\delta_x$ is bounded in $L^2(0,\infty;H^1(\Omega))$ and equation \eqref{eq:delta} gives $\pa_t u^\delta_x $ bounded in $L^2(0,\infty; H^{-2}(\Omega))$.
It follows that
\begin{equation}\label{eq:udeltax} 
u^\delta_x \to u_x \mbox{ strongly in $L^p((0,\infty)\times\Omega)$, for all $p\geq 1$.}
\end{equation}

This shows that 
$$ T^\delta \rightharpoonup T  = \left\{
\begin{array}{ll} 
u u_x u _{xxx} & \mbox{  in } \{u>0\}\\
0 &   \mbox{  in } \{u=0\}
\end{array}
\right. \mbox{ weakly in $L^2$.}
$$
and allows us to pass to the limit in \eqref{eq:tfe5}.
\medskip

It order to get the formula \eqref{eq:vphiT} for $T$, we first write (for $\vphi \in \mathcal D(\overline \Omega)$)
\begin{equation}\label{eq:vphiTdelta} 
\langle T^\delta , \vphi\rangle_{\mathcal D',\mathcal D}   =  \int_\Omega u^\delta (u^\delta_{xx})^2 \vphi \, dx - \frac 5 {6}\int_\Omega {u_x^\delta}^3 \vphi' \, dx- \frac 1 2 \int_\Omega u^\delta {u^\delta_x}^2 \vphi''\, dx.
\end{equation}
This equality follows from the fact that $u^\delta u^\delta_{xxx}=(u^\delta u^\delta_{xx})_x  - u^\delta_x u^\delta_{xx} $, and makes use of a couple of integration by parts.
In particular, it uses the fact that  the boundary terms
$$
u^\delta u^\delta_{x}u^\delta_{xx} \vphi \big|_{\pa\Omega}, \; \frac1 3 {u^\delta_x}^3 \vphi|_{\pa\Omega} \mbox{ and  } \frac 1 2 u^\delta {u^\delta_x}^2 \vphi'|_{\pa\Omega} 
$$
vanish, due to the Neumann boundary conditions and the regularity of $u^\delta$.

We can pass to the limit in the last two terms of \eqref{eq:vphiTdelta}  thanks to \eqref{eq:udeltax}. For the first term, we note that  \eqref{eq:entropyH}  implies $(u^\delta_{xx})^2 $ is bounded in $L^1$ and thus converges to a measure $\mu$ in $[C^0]'$.
Since $u^\delta$ converges uniformly to $u$ we get
$$\int_0^\infty \int_\Omega u^\delta (u^\delta_{xx})^2 \phi(t,x) \, dx \, dt\to \int_0^\infty \int_\Omega u \mu  \phi(t,x) \, dx\, dt.$$
The  bound on the dissipation of energy \eqref{eq:energydelta} shows that $\mu = (u_{xx})^2$  in $\{u>0\}$ which allows us to conclude that
\begin{equation}\label{eq:limdfgt}
\int_0^\infty\int_\Omega u^\delta (u^\delta_{xx})^2 \phi(t,x) \, dx \, dt  \to\int_0^\infty \int_\Omega u (u_{xx})^2 \phi(t,x) \, dx \, dt
\end{equation}
for all $\phi \in \mathcal D([0,\infty)\times\overline \Omega)$.
We can thus pass to the limit in \eqref{eq:vphiTdelta} and get \eqref{eq:vphiT}.
\medskip

Finally, integrating \eqref{eq:tfe5}  and 
using the fact that $u^\delta$ is smooth and satisfies the boundary condition in \eqref{eq:delta}, 
we obtain
$$
\int_\Omega \rho^\delta(t,x)\,dx-\int_0^t\int_\Omega u^\delta | u_{xx}^\delta|^2 \,dxdt=\int_\Omega \rho_{in}^\delta(x)\,dx.
$$
and we can pass to the limit in this equality using \eqref{eq:limdfgt} to get \eqref{eq:entropy0}.

\end{proof}

\bibliographystyle{plain}
\bibliography{Bibliography}
{
  \bigskip
  \footnotesize

 M.G.~Delgadino, \textsc{Department of Mathematics, Imperial College London, London SW7 2AZ, UK}\par\nopagebreak
  \textit{E-mail address:}\texttt{m.delgadino@imperial.ac.uk}

  \medskip

  A.~Mellet, \textsc{Department of Mathematics,
University of Maryland,
College Park, MD 20742
USA}\par\nopagebreak
  \textit{E-mail address:} \texttt{mellet@math.umd.edu}

}

\end{document}